\documentclass{amsart}

\allowdisplaybreaks

\usepackage{hyperref}

\usepackage{nicematrix}

\usepackage{multicol}

\usepackage{graphicx}

\usepackage{xcolor}

\usepackage{amsmath, amssymb, amsfonts, amsthm, fullpage}
\usepackage{graphics,graphicx}
\usepackage{accents}
\usepackage{color}

\usepackage{algorithm}
\usepackage{algorithmicx}
\usepackage{algpseudocode}

\usepackage[margin=0.9in]{geometry}

\newcommand{\norma}[1]{{\left\vert\kern-0.25ex\left\vert\kern-0.25ex\left\vert #1
    \right\vert\kern-0.25ex\right\vert\kern-0.25ex\right\vert}}

\newcommand{\diag}{\mathrm{diag}}

\newcommand{\bd}{\mathbf{d}}
\newcommand{\by}{\mathbf{y}}

\newcommand{\tu}{\tilde{u}}
\newcommand{\tw}{\tilde{w}}

\newcommand{\tz}{\tilde{z}}

\newcommand{\ba}{\mathbf{a}}
\newcommand{\bb}{\mathbf{b}}
\newcommand{\bc}{\mathbf{c}}
\newcommand{\bo}{\mathbf{0}}
\newcommand{\btf}{\mathbf{f}}

\newcommand{\bx}{\mathbf{x}}
\newcommand{\bF}{\mathbf{F}}

\newcommand{\bu}{\mathbf{u}}
\newcommand{\bw}{\mathbf{w}}
\newcommand{\bz}{\mathbf{z}}

\newcommand{\bA}{\mathbf{A}}
\newcommand{\bB}{\mathbf{B}}

\newcommand{\bM}{\mathbf{M}}

\newcommand{\bI}{\mathbf{I}}
\newcommand{\bX}{\mathbf{X}}
\newcommand{\bbeta}{\boldsymbol{\beta}}
\newcommand{\bgamma}{\boldsymbol{\gamma}}

\newcommand{\bzeta}{\boldsymbol{\zeta}}

\newcommand{\tbn}[1]{{\left\vert\kern-0.25ex\left\vert\kern-0.25ex\left\vert #1 \right\vert\kern-0.25ex\right\vert\kern-0.25ex\right\vert}}

\newtheorem{remark}{Remark}[section]
\newtheorem{lemma}{Lemma}[section]
\newtheorem{proposition}{Proposition}[section]
\newtheorem{theorem}{Theorem}[section]

\AtBeginDocument{%
  \setlength{\belowdisplayskip}{5pt} \setlength{\belowdisplayshortskip}{5pt}
  \setlength{\abovedisplayskip}{5pt} \setlength{\abovedisplayshortskip}{5pt}
}

\begin{document}

\title[A dichotomy of finite element spaces and its application to an energy-conservative scheme for the regularized long wave equation]{A dichotomy of finite element spaces and its application to an energy-conservative scheme for the regularized long wave equation}

\author{Dimitrios Antonopoulos}
\address{\textbf{D.~Antonopoulos:} Department of Mathematics, University of Athens, 15784 Zographou, Greece}
\email{antonod@math.uoa.gr}

\author{Dimitrios Mitsotakis}
\address{\textbf{D.~Mitsotakis:} Victoria University of Wellington, School of Mathematics and Statistics, PO Box 600, Wellington 6140, New Zealand}
\email{dimitrios.mitsotakis@vuw.ac.nz}




\date{\today}



\begin{abstract}
Certain energy-conservative Galerkin discretizations for nonlinear dispersive wave equations have revealed an unusual convergence behavior: optimal convergence is attained when continuous Lagrange finite element spaces of odd polynomial degree are employed, whereas the use of even-degree polynomials leads to reduced convergence rate. The present work demonstrates that this behavior is intrinsic to the structure of the finite element spaces themselves. In particular, it is shown to be closely connected to the standard $L^2$-projection of derivatives, which possesses a super-approximation property exclusively for odd polynomial degrees. We also examine the implications of this feature for an energy-conservative Galerkin approximation of the regularized long-wave equation where the energy is a cubic functional. Although the resulting scheme conserves both mass and energy, we further show that the impulse is approximated with high accuracy, and we establish {\em a priori} error bounds for the associated semi-discrete formulation.
\end{abstract}

\maketitle

\section{Introduction}\label{sec:intro}

This work primarily establishes a property of one-dimensional continuous Lagrange finite element spaces that exhibits a fundamental distinction between odd- and even-degree piecewise polynomials. Let $\mathcal{V}_h^k$ denote the standard Lagrange finite element space of continuous, periodic, piecewise $k$-degree polynomial functions over a grid with uniform mesh length $h$. For a sufficiently smooth function $u$, the stability of the $L^2$-projection together with its standard error estimate in the $L^2$-norm implies that
\begin{equation}\label{eq:standard}
\|P[(Pu-u)_x]\|\leq \|(Pu-u)_x\|\leq Ch^k\ ,
\end{equation}
for some constant $C > 0$ independent of $h$, where $P: L^2 \to \mathcal{V}_h^k$ is the usual $L^2$-projection onto the finite element space.
However, this bound is not optimal. In this work, we prove that if $u$ is a periodic function in $C^{k+2}$ for $k=2\kappa+1$ with $\kappa=0,1,2,\dots$, then $\|P[(Pu-u)_x]\|=O(h^{k+1})$.  In addition, when $k=1$ and $u$ is periodic in $C^{5}$, the sharper estimate $\|P[(Pu-u)_x]\| = O(h^{4})$ holds. By contrast, for even degrees $k=2\kappa$ with $\kappa=1,2,\dots$, the estimate (\ref{eq:standard}) remains optimal, and no further improvement can be achieved.
The analysis presented in this work is carried out under periodic boundary conditions, although the main results remain valid under Dirichlet boundary conditions. While the distinction between odd- and even-degree spaces persists under Neumann boundary conditions, the corresponding error for $k=2\kappa+1$, $\kappa=0,1,\dots$ is apparently of $O(h^{k+1/2})$ and warrants further investigation.

The approximation property described above is not restricted to the one-dimensi\-onal setting. The same theoretical arguments extend naturally to two-dimensi\-onal tensor-product Lagrange finite element spaces defined over rectangular grids, where an analogous distinction between odd and even polynomial degrees is expected. However, numerical experiments indicate that this behavior may not persist on general triangulations. In particular, \cite{AM2025} reports that for a conservative finite element method applied to two-dimensional Boussinesq systems similar to the one presented here, the convergence rate remains optimal and independent of the polynomial degree of the finite element space. On the other hand, this property explains the observations of \cite{MRKS2021}. That work introduced a conservative method preserving a cubic energy functional, and optimal convergence was observed only for odd-degree spaces. Since the same mechanism operates in cubic energy-conserving schemes, we demonstrate its relevance by presenting an energy-conserving finite element method for the regularized long wave (RLW) equation, also known as the Benjamin--Bona--Mahony (BBM) equation, which serves as a fundamental model in this class of problems.

The RLW equation was originally derived for the study of long surface water waves of small amplitude over a flat bottom, as an approximation of the Euler equations in water wave theory and as a regularization of the Korteweg--de Vries (KdV) equation \cite{BBM1972,Pere1966,KDV1895,BS1975}. In dimensionless but unscaled variables, the RLW equation can be written in the form
\begin{equation}\label{eq:bbm}
u_t + u_x + u u_x - u_{xxt} = 0\ ,
\end{equation}
where $u=u(x,t)$ denotes the free surface elevation of the water from its rest position, while $x$ and $t$ are the spatial and temporal independent variables.

Ideally, a water wave equation should not only possess well-posedness properties, asymptotic justification, and the existence of classical solitary waves, but also be applicable to practical problems and satisfy fundamental conservation laws such as mass and energy. For flat-bottom topography in a non-dissipative continuous medium, it is common to also consider impulse (momentum) conservation. Impulse is typically represented by a quadratic functional, which remains invariant when the physical system is translation-invariant in the $x$-direction \cite{Benjamin1984}. The RLW equation is known to admit only three independent and non-trivial conservation laws: the mass, impulse and energy conservation laws \cite{Olver1979}.  

Multi-symplecticity of some nonlinear and dispersive wave equations is another powerful structure directly related to these conserved quantities \cite{DDM2019i,DDM2019ii}. The RLW equation is also equipped with a multi-symplectic structure \cite{CWG1991,SQ2004,Cai2009,Cai2009b}. Due to the uniqueness of the three fundamental conservation laws of the RLW equation, its multi-symplectic conservation law happens to be equivalent to impulse conservation. However, since the impulse functional is no longer preserved when the bottom topography is uneven \cite{I2010,AM2025}, energy and mass conservation remain the only reliable measures that can serve as high-accuracy indicators in long-term numerical simulations \cite{IKKM2021,AM2025,AK2014,KKM2016}. 

In this work, we introduce a modified finite element method that preserves both a cubic energy functional and mass. When odd-degree polynomials are employed, the numerical solutions converge with the optimal order $O(h^{k+1})$, whereas for even-degree polynomials the convergence is suboptimal, of order $O(h^k)$, reflecting the inherent dichotomy of the finite element spaces. Moreover, the approximation of the derivative $u_x$ exhibits the same order of convergence as $u$, revealing a notable super-approximation property. Although the proposed modified Galerkin method does not preserve the impulse functional, it approximates it with high accuracy. Other similar Galerkin finite element methods for the RLW equation are presented in \cite{GH2006}.

The structure of the manuscript is the following: In Section \ref{sec:dichotomy} we prove the approximation property of the $L^2$-projection that bisects the Lagrange finite element spaces.  How this property affects the convergence of an energy-conservative Galerkin method for the RLW equation is shown in Section \ref{sec:galerkin}. The conclusions are presented in Section \ref{sec:conclusions}.

\section{A dichotomy of finite element spaces}\label{sec:dichotomy}

\subsection{Approximation spaces}

For the purposes of the analysis, and for the sake of simplicity, we consider the interval $I=[0,1]$. For any $N>0$, let $h=1/N$ and define the partition
$\Delta=\{x_0,x_1,\dots, x_N\}$ of $[0,1]$
where $0=x_0<x_1<\cdots<x_N=1$ and $x_i=ih$ for each $i=0,1,\dots,N$. 
This partition is periodically extended to a partition of $\mathbb{R}$ by setting $x_{jN+s}=x_s$ for all $j\in\mathbb{Z}$ and $s=0,1,\dots,N-1$.  
Let $k\in\mathbb{Z}$ with $k\geq 1$. We define the Lagrange finite element space
$$\mathcal{V}_h^k=\{\phi\in C_p~:~ \phi|_{[x_{i-1},x_i]}\in\mathbb{P}_k \quad \text{for}\quad i=1,2,\dots,N\}\ ,$$
where $C_p=\{f\in C(\mathbb{R})~:~ f(x)=f(x+1)\quad \text{for all}~x\in\mathbb{R}\}$, and $\mathbb{P}_k$ denotes the space of polynomials of degree at most $k$. In what follows, $C_p^k$ is the set of functions $u$ in $C_p$ such that the $j$-derivative $u^{(j)}$ is also in $C_p$ for all $j=1,\dots,k$. 

In this work, we employ the Sobolev space $H^s_p$, consisting of $1$-periodic functions in $H^s$, equipped with the norm $\|\cdot\|_s$ evaluated over one period. We also use the space $L^2(0,1)=H^0(0,1)$ with the standard $L^2$-norm induced by the inner product $(\cdot,\cdot)$. In addition, $\|\cdot\|_\infty$ will denote the usual $L^\infty$-norm.

Our proof relies on a direct-sum decomposition $\mathcal{V}_h^k=\mathcal{S}_h^1\oplus \mathcal{S}_h^2$, following closely the construction in \cite{CT1987}. We define
$$\mathcal{S}_h^2=\{\phi\in \mathcal{V}_h^k~:~ \phi(x_i)=0 \quad \text{for all}\quad i=1,2,\dots, N\}\ ,$$
and set $\mathcal{S}_h^1=(\mathcal{S}_h^2)^\perp$, the $L^2$-orthogonal complement of $\mathcal{S}_h^2$. 
The space $\mathcal{S}_h^1$ consists of \emph{vertex} functions, whereas $\mathcal{S}_h^2$ comprises \emph{bubble} functions. These correspond to the standard decomposition of Lagrange finite element spaces into vertex-based and interior (bubble) shape functions; see \cite{Solin}.

Observe that when $k = 1$, we have $\mathcal{S}_h^2 = {0}$, as expected. In this case, the linear finite element space admits no bubble functions, and its basis is formed exclusively by vertex functions, namely the standard hat functions.
Let $\pi_1$ and $\pi_2$ denote the corresponding $L^2$-projections onto $\mathcal{S}_h^1$ and $\mathcal{S}_h^2$, respectively. The associated $L^2$-projection on $\mathcal{V}_h^k$ is then given by $P=\pi_1+\pi_2$.

The projection $\pi_2$ is defined locally on each interval $I_i=(x_{i-1},x_i)$ for $i=1,2,\dots,N$, by requiring that for any $v\in L^2$,
$$(\pi_2 v,q)_{I_i}=(v,q)_{I_i} \quad \text{for all} \quad q\in \mathbb{P}_{k}^{i,0}=\{q\in\mathbb{P}_k~:~q(x_{i-1})=q(x_i)=0\}\ ,$$
where $(\cdot,\cdot)_{I_i}$ denotes the usual inner product on $L^2(I_i)$.

Any $p\in \mathcal{S}_h^2$ restricted in $I_i$ can be written in the form $p(x)=(x-x_i)(x-x_{i-1})s(x)$, where $s\in \mathbb{P}_{k-2}(I_i)$. The space $\mathbb{P}_{k-2}(I_i)$ has dimension $k-1$, and because the interval $I$ consists of $N$ sub-intervals, we conclude that the subspace $\mathcal{S}_h^2$ has dimension $(k-1)N$. 
A basis of $\mathcal{S}_h^2$ may be constructed using the functions $\phi_j:[0,1]\to \mathbb{R}$ defined by
$$\phi_j(x)=x(1-x)b_{k-2,\,j-1}(x),\quad j=1,2,\dots,k-1\ ,$$
where
$$b_{k-2,\,j-1}(x)=\binom{k-2}{j-1} x^{j-1}(1-x)^{k-1-j}, \quad j=1,2,\dots, k-1\ ,$$
are the Bernstein polynomials of degree $k-2$. These polynomials are linearly independent and span the space $\mathbb{P}_{k-2}(0,1)$.

The dimension of $\mathcal{S}_h^1$ is $\dim \mathcal{S}_h^1=\dim \mathcal{V}_h^k-\dim \mathcal{S}_h^2=N$, and each function in this space is determined by its values at the nodes $x_i$. To construct a basis for this space, we make use of the function
$$\psi(x)=\frac{(-1)^{k-1}}{k!} \frac{1}{x(1-x)}\frac{d^{k-1}}{dx^{k-1}}\left[x^{k+1}(1-x)^k \right],\quad x\in[0,1]\ ,$$
which satisfies $\psi(0)=0$, $\psi(1)=1$, and
\begin{equation}\label{eq:i}
    \int_0^1\psi(x)p(x)~dx=0\quad \text{for all}\quad p\in \mathbb{P}_k^0=\{p\in \mathbb{P}_k~:~p(0)=p(1)=0\}\ .
\end{equation}
The construction of $\psi$ relies on the following characterization theorem for orthogonal polynomials; see \cite[Th. 12.5, p. 141]{P1981}:
\begin{theorem}\label{thm:powel}
Let $w:[a,b]\to \mathbb{R}$ be a continuous function. A function $v\in C[a,b]$ satisfies the orthogonality conditions
$$\int_a^bw(x)v(x)p(x)~dx=0,\quad p\in \mathbb{P}_m\ ,$$
if and only if there exists a function $f\in C^{m+1}(a,b)$ such that for all $x\in[a,b]$, 
$w(x)v(x)=f^{(m+1)}(x)$, 
with 
$f^{(j)}(a)=f^{(j)}(b)=0$, $j=1,2,\dots, m$.
\end{theorem}

In addition, we will need the following properties of $\psi$:
\begin{lemma}\label{lem:usefl}
If $k=2\kappa+1$, $\kappa=0,1,\dots$, then
\begin{align}
&\int_0^1 \psi^2(x)~dx=\frac{1}{k(k+2)}\ , \tag{i}\label{eq:ii} \\ 
&\int_0^1 \psi(x)\psi(1-x)~dx=\frac{1}{k(k+1)(k+2)}\ , \tag{ii}\label{eq:iii}\\
&\int_0^1 x^j\psi(x)~dx=\frac{1}{k(k+2)},\quad j=1,2,\dots, k\ , \tag{iii}\label{eq:iv}\\
&\int_0^1 \psi(x)~dx=\frac{1}{k(k+1)}\ , \tag{iv}\label{eq:v}\\
&\int_0^1 (1-x)^j\psi(x)~dx = \frac{1}{k(k+1)(k+2)},\quad j=1,2,\dots, k\ , \tag{v}\label{eq:vi}\\
&\psi'(0)=\frac{k+1}{2},\quad \psi'(1)=\frac{k^2+2k-1}{2}\ ,\tag{vi}\label{eq:vii}\\
& \int_0^1 \psi'(x)\psi(1-x)~dx=\frac{1}{k+1}\ .\tag{vii}\label{eq:viii}
\end{align}
\end{lemma}
\begin{proof}
The equalities (\ref{eq:ii}), (\ref{eq:iii}) are established in \cite{CT1987}, as well as (\ref{eq:iv}) for $j=1$. Property (\ref{eq:iv}) can be extended for $j=2,3,\dots, k$ by induction, using the relation $x^j=x^{j-1}(x-1)+x^{j-1}$ together with (\ref{eq:i}). 

Since the function $\psi(x)+\psi(1-x)-1$ belongs to $\mathbb{P}_k^0$, it follows from (\ref{eq:i}) that
$$\int_0^1 [\psi(x)^2+\psi(x)\psi(1-x)-\psi(x)]~dx=\int_0^1 \psi(x)(\psi(x)+\psi(1-x)-1)~dx=0\ ,$$
and, combined with (\ref{eq:ii}) and (\ref{eq:iii}), this yields (\ref{eq:v}).
Using the identity $(1-x)^j=(1-x)^{j-1}-(1-x)^{j-1}x$, together with (\ref{eq:i}), equation (\ref{eq:vi}) can be established by induction. 

From the definition of $\psi$, and using the Leibniz rule for the $(k-1)$-st derivative of a product, we obtain
$$\psi(x)=\sum_{j=0}^{k-1}{{k-1}\choose{j}}\frac{(k+1)!(-1)^j}{(k+1-j)!(j+1)!}x^{k-j}(1-x)^j\ ,$$
which, for $k=1$ gives $\psi(x)=x$, and for odd $k\geq 3$, yields
$$\psi(x)=x^k-\frac{(k-1)(k+1)}{2}x^{k-1}(1-x)+s(x)+\frac{k+1}{2}x(1-x)^{k-1}\ ,$$
where $s\in \mathbb{P}_k^0$ satisfies $s'(0)=s'(1)=0$. Thus,
$$\psi'(x)=kx^{k-1}+\frac{(k-1)(k+1)}{2}x^{k-1}+\widetilde{s}(x)+\frac{k+1}{2}(1-x)^{k-1}\ ,$$
with $\widetilde{s}\in \mathbb{P}_k^0$, and therefore (\ref{eq:vii}) follows.

The function $s:[0,1]\to \mathbb{R}$, $s(x)=\psi'(x)-x\psi'(1)-(1-x)\psi'(0)$ belongs to $\mathbb{P}_{k-1}^0$, and thus
$\int_0^1 s(x)\psi(1-x)~dx=0$,
which yields
$$\int_0^1 \psi'(x)\psi(1-x)~dx=\psi'(1)\int_0^1 x\psi(1-x)~dx+\psi'(0)\int_0^1(1-x)\psi(1-x)~dx\ .$$
Finally, property (\ref{eq:viii}) follows from the right-hand side of the above equality together with (\ref{eq:iv}), (\ref{eq:vi}) and (\ref{eq:vii}).
\end{proof}

\begin{remark}\label{rmk:onepone}
For any even integer $k$, the right-hand side of (\ref{eq:iii}) is $-1/(k(k+1)(k+2))$, and the right-hand side of (\ref{eq:v}) is $1/((k+1)(k+2))$, while $\psi'(0)=-(k+1)/2$. 
\end{remark}

We consider the functions
{\footnotesize 
$$L(x)=\left\{
\begin{array}{ll} 
\psi(1-|x|), & |x|\leq 1\\
0, & |x|>1
\end{array}
\right. ,\quad
Q_j(x)=\left\{\begin{array}{ll}
\phi_j(x), &  x\in[0,1]\\
0,  &\text{otherwise}
\end{array}
\right. ,\quad j=1,2,\dots, k-1\ , 
$$
}
and 1-periodic functions of the form
$$w(x)=\sum_{m\in\mathbb{Z}} \left[d_mL(h^{-1}x-m)+\sum_{\nu=1}^{k-1} c_{m,\nu}Q_\nu(h^{-1}x-m) \right]\ ,$$
where, for fixed $h>0$, the coefficients satisfy $d_{m+N}=d_m$ and $c_{\nu,m+N}=c_{\nu,m}$. 
Alternatively, such a function can be written in the form
$$w(x)=\sum_{i=1}^N\left[ d_i \ell_i(x)+\sum_{\nu=1}^{k-1}c_{i,\nu} q_{i,\nu}(x)\right]\ ,$$
where, for any $j\in\mathbb{Z}$ 
$$\ell_j(x)=\sum_{m\in\mathbb{Z}} L(h^{-1}x-mN-j),\quad q_{j,\nu}(x)=\sum_{m\in\mathbb{Z}} Q_\nu(h^{-1}x-mN-j+1)\ .$$
%
It is straightforward to verify that $\ell_i(x_j)=\delta_{ij}$, where $\delta_{ii}=1$ and $\delta_{ij}=0$ for $i\not=j$. 
Finally, for $i=1,2,\dots, N$ and $\nu=1,2,\dots, k-1$, we have
$$q_{i,\nu}(x)=Q_\nu\left(\frac{x-x_{i-1}}{h} \right)\ .$$
The functions $\ell_i$, $i=1,2,\dots, N$, are linearly independent and span $\mathcal{S}_h^1$. For instance, if $a_i\in \mathbb{R}$ satisfy
$\sum_{i=1}^N a_i \ell_i(x)=0\quad \text{for all}\quad x\in[0,1]$,
then evaluating at $x=x_j$ gives $a_j=0$. If $V=\mathrm{span}(\ell_1,\ell_2,\dots,\ell_N)$, we have that $V\subset \mathcal{S}_h^1$, and since $V$ and $\mathcal{S}_h^1$ have the same dimension, they coincide.

\begin{remark}\label{rmk:qim}
Note that if $k>3$, then $q_{i,\nu}'\in \mathcal{S}_h^2$ for $\nu=2,\dots, k-2$.
\end{remark}

\subsection{Estimation of the projections $\pi_1$, $\pi_2$}

Let $k=2\kappa+1$, $\kappa=0,1,\dots$ be an odd integer and $v\in L^2$ be a 1-periodic function. From the previous analysis, we write 
\begin{equation}\label{eq:repr1}
(\pi_1 v)(x)=\sum_{i=1}^Nd_i \ell_i(x),\qquad x\in [0,1]\ ,
\end{equation}
 where the vector $\bd=(d_1,d_2,\dots,d_N)^T$ is the solution of the $N\times N$ linear system $\bA\bd=\bb$, where $\bb=(b_1,b_2,\dots, b_N)^T$ with $b_i=(v,\ell_i)$, and $\bA$ is the Gram matrix with entries $a_{ij}=(\ell_j,\ell_i)$, for $1\leq i,j\leq N$. 

From (\ref{eq:ii}) and (\ref{eq:iii}) of Lemma \ref{lem:usefl} we have
$$
\bA=\frac{h}{k(k+1)(k+2)}{\rm Circulant}(2k+2, 1, 0, \dots, 0, 1)\ ,
$$
where ${\rm Circulant}(2k+2, 1, 0, \dots, 0, 1)$ denotes the circulant matrix whose first row is given by the entries in the argument. This matrix is symmetric and positive definite.
%
%
%
Thus, there exist constants $\sigma_1>0$ and $\sigma_2>0$, depending on $k$ but independent of $h$, such that
\begin{equation}\label{eq:twopone}
\sigma_1 h|\bbeta|^2\leq \langle \bA\bbeta,\bbeta\rangle \leq \sigma_2 h|\bbeta|^2 \quad \text{for all}\quad \bbeta\in \mathbb{R}^N\ ,
\end{equation}
where $\langle\cdot,\cdot\rangle$ denotes the standard inner product on $\mathbb{R}^N$, and $|\cdot|$ the corresponding norm.

Indeed, from \cite{TW1974}, if $\alpha_1=a_{11}$, $\alpha_2=a_{12}$, and $\alpha_N=a_{1N}$, then the eigenvalues $\lambda_m$, $m=1,2,\dots, N$, are given by evaluating the polynomial $\alpha(z)=\alpha_1+\alpha_2 z+\alpha_N z^{N-1}$ at $z=\exp(2\pi i m/N)$. Consequently, we obtain
$\lambda_m=\frac{2h}{k(k+1)(k+2)}\left(k+1+\cos\left(\frac{2\pi m}{N}\right)\right)$,
which completes the proof of (\ref{eq:twopone}).  

Furthermore, we can write
$\bA=\frac{2h}{k(k+2)}\left(\bI_N+\widetilde{\bA}\right)$,
where $\bI_N$ is the $N\times N$ identity matrix. Since $\|\widetilde{\bA}\|_\infty=1/(k+1)<1$, where $\|\cdot\|_\infty$ denotes the maximum norm for square matrices, it follows that
$$\|\bA^{-1}\|_\infty\leq \frac{(k+1)(k+2)}{2h}\ .$$
Thus, for $\bgamma\in\mathbb{R}^N$ and $\alpha\in\mathbb{Z}$, we have
\begin{equation}\label{eq:twoptwo}
|\bgamma|_\infty=O\left(h^\alpha\right)\quad\text{if and only if}\quad |\mathbf{A}\bgamma|_\infty=O\left(h^{\alpha+1}\right)\ ,
\end{equation}
where $|\cdot|_\infty$ denotes the maximum norm on $\mathbb{R}^N$.

Similarly, if we write
\begin{equation}\label{eq:repr2}
(\pi_2 v)(x)=\sum_{i=1}^N\sum_{m=1}^{k-1} c_{i,m}q_{i,m}(x),\quad x\in [0,1]\ ,
\end{equation}
then the vector $\bc_i=(c_{i,1},c_{i,2},\dots,c_{i,k-1})^T$ is the solution of the system
$\bB \bc_i=\bgamma_i$,
where $\bgamma_i\in \mathbb{R}^{k-1}$ has entries $(\gamma_i)_\nu=(v,q_{i,\nu})$, $\nu=1,2,\dots,k-1$, and $\bB$ is the $(k-1)\times(k-1)$ matrix with entries $b_{m,\nu}=(q_{i,\nu},q_{i,m})_{I_i}=h(\phi_\nu,\phi_m)$. The matrix $\bB$ is symmetric and positive definite. Therefore, there exist constants $\sigma_3>0$ and $\sigma_4>0$, depending on $k$ but independent of $h$, such that 
\begin{equation}\label{eq:twopthree}
    \sigma_3 h |\bgamma|^2\leq \langle \bB \bgamma,\bgamma \rangle\leq \sigma_2 h|\bgamma|^2\quad \text{for all}\quad \bgamma\in \mathbb{R}^{k-1}\ . 
\end{equation}

We are now in a position to estimate both $\pi_1$ and $\pi_2$. 

\begin{lemma}\label{lem:estimatepis}
Let $k=2\kappa+1$, $\kappa=0,1,\dots$ and let $v\in L^2$ be a 1-periodic function, and let $C_1$ and $C_2$  be positive constants independent of $h$. Then, for $\alpha>0$ and $\beta>0$, the following hold: 
\begin{equation}\label{eq:estpi1}
\text{if}\quad \max_{1\leq i\leq N}|(v,\ell_i)|\leq C_1 h^\alpha \quad \text{for some} \quad \alpha>0,\quad \text{then} \quad \|\pi_1 v\|\leq \widetilde{C}_1 h^{\alpha-1} \ ,
\end{equation}
and
\begin{equation}\label{eq:estpi2}
\text{if}\quad \max_{\substack{1\leq i\leq N\\1\leq m\leq k-1}}|(v,q_{i,m})|\leq C_2 h^\beta \quad \text{for some} \quad \beta>0,\quad \text{then} \quad \|\pi_2 v\|\leq \widetilde{C}_2 h^{\beta-1} \ ,
\end{equation}
for some positive constants $\widetilde{C}_1$ and $\widetilde{C}_2$ independent of $h$.
\end{lemma}
\begin{proof}
The representation (\ref{eq:repr1}) of $\pi_1 v$  yields
\begin{equation}\label{eq:normsqrd}
\|\pi_1v\|^2=(v,\pi_1 v)=\sum_{i=1}^N (v,\ell_i) d_i=\langle \bb,\bd\rangle\ .
\end{equation} 
Using the Cauchy-Schwarz inequality together with the hypothesis in (\ref{eq:estpi1}), we have
$$\|\pi_1 v\|^2\leq C_1 h^{\alpha}N^{1/2}\left(\sum_{i=1}^N d_i^2 \right)^{1/2}=C_1 h^{\alpha-1}\sqrt{h}\left(\sum_{i=1}^N d_i^2 \right)^{1/2}\ .$$
Therefore, if we denote $\widetilde{C}_1=C_1/\sigma_1$, where $\sigma_1$ is defined in (\ref{eq:twopone}), then using (\ref{eq:twopone}) and (\ref{eq:normsqrd}), we obtain
$$\|\pi_1 v\|^2\leq \widetilde{C}_1 h^{\alpha-1} \sqrt{\langle \bA\bd,\bd\rangle}=\widetilde{C}_1 h^{\alpha-1} \sqrt{\langle \bb,\bd\rangle}= \widetilde{C}_1 h^{\alpha-1} \|\pi_1 v\|\ ,$$
which gives the estimate in (\ref{eq:estpi1}).

Similarly, from (\ref{eq:repr2})
$$(v,\pi_2 v)=\sum_{m=1}^{k-1}\sum_{i=1}^N (v,q_{i,m})c_{i,m}\ .$$
Because
$$\left(\sum_{i=1}^N (v,q_{i,m})c_{i,m} \right)^2\leq \sum_{i=1}^N|(v,q_{i,m})|^2\sum_{i=1}^N|c_{i,m}|^2\leq C_2^2 h^{2\beta-1}\sum_{i=1}^N|c_{i,m}|^2\ ,$$
we obtain
$$\|\pi_2 v\|^2\leq C_2 h^{\beta-1/2}\sum_{m=1}^{k-1}\left(\sum_{i=1}^N|c_{i,m}|^2 \right)^{1/2}=C_2h^{\beta-1}\sum_{m=1}^{k-1}\sqrt{h}\left(\sum_{i=1}^N|c_{i,m}|^2 \right)^{1/2}\ .$$
The estimate in (\ref{eq:estpi2}) then follows from this inequality with the help of (\ref{eq:twopthree}).
\end{proof}

It is worth noting that the previous results related to $\bA$, $\bB$ and Lemma \ref{lem:estimatepis} remain valid for even values of $k$, but the inequality constants differ. However, the following error estimate for $\pi_1v$ is only true for odd values of $k$.

\begin{lemma}\label{lem:errestpi1}
Let $k=2\kappa+1$, $\kappa=0,1,\dots$. Then there exists a constant $C>0$, independent of $h$, such that for any function $v\in C_p^{k+1}$:
\begin{equation}\label{eq:errestpi1}
\max_{1\leq i\leq N}|(\pi_1 v-v)(x_i)|\leq Ch^{k+1}\|v^{(k+1)}\|_\infty\ .
\end{equation}
\end{lemma}
\begin{proof}
Let $\bgamma\in \mathbb{R}^N$ be a vector whose entries are given by $\gamma_i=(\pi_1v)(x_i)-v(x_i)$ for $i=1,2,\dots,N$. Then the entries of the  product $\bA\bgamma$ are
\begin{equation}\label{eq:gammai1}
(\bA\bgamma)_i=(v,\ell_i)-\frac{h}{k(k+1)(k+2)}[(2k+2)v_i+v_{i+1}+v_{i-1}], \quad 1\leq i \leq N\ ,
\end{equation}
where $v_i=v(x_i)$, for $i=1,2,\dots, N$, and $v_0=v_N$ and $v_{N+1}=v_1$.

On the other hand, 
$$\begin{aligned}
(v,\ell_i) &=\int_{x_{i-1}}^{x_i} v(x)\ell_i(x)~dx+\int_{x_i}^{x_{i+1}}v(x)\ell_i(x)~dx\\
&=\int_{x_{i-1}}^{x_i} v(x)\psi\left(\frac{x-x_{i-1}}{h}\right)~dx+\int_{x_i}^{x_{i+1}}v(x)\psi\left(\frac{x_{i+1}-x}{h}\right)~dx\\
&= h\int_0^1 \left[ v(x_i-h(1-x))+v(x_i+h(1-x)) \right]\psi(x)~dx\ .
\end{aligned}$$

If we define $w_i(x)=v(x_i-h(1-x))+v(x_i+h(1-x))$ and we denote by $v_i^{(j)}$ the $j$-th derivative of $v$ at $x_i$, then using Taylor expansions of $v$ about $x_i$ we obtain
$$w_i(x)=2v_i+\frac{2h^2}{2!}(1-x)^2v_i''+\cdots + \frac{2h^{k-1}}{(k-1)!}(1-x)^{k-1}v_i^{(k-1)}+h^{k+1} r_{i,k}(x;h)\ ,$$
where the residual $r_{i,k}(x;h)$ satisfies $\max_{0\leq x\leq 1}|r_{i,k}(x;h)|\leq C_k\|v^{(k+1)}\|_\infty$, for some positive constant $C_k$ depending on $k$ but independent of $h$. The dependence of $r_{i,k}$ on $h$ is due to the presence of $h$ inside the argument of $v^{(k+1)}$.

Therefore,
\begin{equation}\label{eq:lastintegr}
\text{\footnotesize
$\begin{aligned}
(v,\ell_i) &=h\int_0^1 \left[2v_i+\frac{2h^2}{2!}v_i''(1-x)^2+\cdots+\frac{2h^{k-1}}{(k-1)!}v_i^{(k-1)}(1-x)^{k-1} \right]\psi(x)~dx\\
& \quad +h^{k+2}\int_0^1 r_{i,k}(x;h)\psi(x)~dx\ .
\end{aligned}$
}
\end{equation}
Using the properties (\ref{eq:vi}) and (\ref{eq:v}) of Lemma \ref{lem:usefl} for $\psi$, we then have
\begin{equation}\label{eq:twopeight}
\text{\footnotesize
    $(v,\ell_i)=h\left[\frac{2}{k(k+1)}v_i+\frac{1}{k(k+1)(k+2)}\left(\frac{2h^2}{2!}v_i''+ \cdots +\frac{2h^{k-1}}{(k-1)!}v_i^{(k-1)}\right) \right] + h^{k+2} \widetilde{r}_{i,k}$}\ ,
\end{equation}
where $\widetilde{r}_{i,k}$ denotes the last integral in (\ref{eq:lastintegr}).

Similarly, using Taylor expansions of $v_{i+1}=v(x_i+h)$ and $v_{i-1}=v(x_i-h)$ about $x_i$, we obtain
\begin{equation}\label{eq:twopnine}
(2k+2)v_i +v_{i+1}+v_{i-1}=2(k+2) v_i+\frac{2h^2}{2!}v_i''+\cdots+\frac{2h^{k-1}}{(k-1)!}v_i^{(k-1)}+h^{k+1}\rho_{i,k}(h)\ ,
\end{equation}
where $|\rho_{i,k}(h)|\leq \widetilde{C}_k\|v^{(k+1)}\|_\infty$ with $\widetilde{C}_k>0$ independent of $h$ and $v$.

Multiplying (\ref{eq:twopnine}) with $h/(k(k+1)(k+2))$ and subtracting from (\ref{eq:twopeight}), shows that for (\ref{eq:gammai1}) there exists a constant $C$, independent of $h$, such that
$|(\bA\bgamma)_i|\leq Ch^{k+2}\|v^{(k+1)}\|_\infty$, $i=1,2,\dots,N$.
Finally, the equivalence of (\ref{eq:twoptwo}) implies $|\bgamma|_\infty\leq Ch^{k+1}\|v^{(k+1)}\|_\infty$ for $i=1,2,\dots,N$, which completes the proof of the estimate (\ref{eq:errestpi1}).
\end{proof}

In addition to the previous error estimate, for odd values of $k$, the projection $\pi_1$ satisfies a super-approximation property.
\begin{lemma}\label{lem:superapprox}
Let $k=2\kappa+1$, $\kappa=0,1,\dots$. Then, there exists a positive constant $C$, independent of $h$, such that for any function $v\in C_p^{k+2}$: 
\begin{equation}\label{eq:superapprox}
\max_{1\leq i\leq N} \left|(\pi_1 v-v,\ell_i)_{I_i} \right|\leq C h^{k+3}\|v^{(k+2)}\|_\infty,\qquad I_i=(x_{i-1},x_i)\ .
\end{equation}
\end{lemma}
\begin{proof}
Writing $\pi_1v=\sum_{i=1}^N d_i\ell_i$ and defining $\beta_i=(\pi_1v-v,\ell_i)_{I_i}$, we have 
$$\begin{aligned}
\beta_i&=\int_{x_{i-1}}^{x_i}\left[d_{i-1}\ell_{i-1}(x)\ell_i(x)+d_i\ell_i^2(x) \right]~dx-(v,\ell_i)_{I_i}\\
&=h\ d_{i-1}\int_0^1 \psi(x)\psi(1-x)~dx+h\ d_i \int_0^1 \psi^2(x)~dx-(v,\ell_i)_{I_i}\ .
\end{aligned}$$
Using the properties (\ref{eq:ii}) and (\ref{eq:iii}) of Lemma \ref{lem:usefl} for $\psi$, this simplifies to
$$\beta_i=\frac{h}{k(k+1)(k+2)}\left[d_{i-1}+(k+1)d_i \right]- (v,\ell_i)_{I_i}\ .$$
Thus, the vector $\bbeta$ with entries $\beta_i$ can be written as
$$\bbeta=\frac{h}{k(k+1)(k+2)}\left[\widetilde{\bd}+(k+1)\bd \right] - \bc\ ,$$
where $\bc$ is the vector with entries $(v,\ell_i)_{I_i}$, and $\widetilde{\bd}$ is the vector with $\widetilde{d}_i=d_{i-1}$.
Because $\bA$ is circulant and $\bd=\bA^{-1}\bb$ with $b_i=(v,\ell_i)$, we then have
$$\bbeta=\frac{h}{k(k+1)(k+2)}\bA^{-1}\left[\widetilde{\bb}+(k+1)\bb \right] - \bc\ ,$$
where $\widetilde{\bb}$ has entries $\widetilde{b}_i=b_{i-1}$.
Consequently, 
$$
\text{\scriptsize
$
\begin{aligned}
(\bA\bbeta)_i&=\frac{h}{k(k+1)(k+2)}\left[(v,\ell_{i-1})+(k+1)(v,\ell_i)-(2k+2)(v,\ell_i)_{I_i}-(v,\ell_{i+1})_{I_{i+1}}-(v,\ell_{i-1})_{I_{i-1}} \right]\\
&=\frac{h}{k(k+1)(k+2)}\left[(v,\ell_{i-1})_{I_i}-(k+1)(v,\ell_i)_{I_i}+(k+1)(v,\ell_i)_{I_{i+1}}-(v,\ell_{i+1})_{I_{i+1}} \right]\\
&=\frac{h}{k(k+1)(k+2)}[(v,(k+1)\ell_i-\ell_{i+1})_{I_{i+1}}-(v,(k+1)\ell_i-\ell_{i-1})_{I_i}]\ .
\end{aligned}
$}
$$
If we denote the terms in the brackets by 
$\gamma_i=(v,(k+1)\ell_i-\ell_{i+1})_{I_{i+1}}-(v,(k+1)\ell_i-\ell_{i-1})_{I_i}$, then
$$\text{\footnotesize $
\gamma_i=h\left(\int_0^1 v(x_i+hx)[(k+1)\psi(1-x)-\psi(x)]~dx-\int_0^1v(x_i-hx)[(k+1)\psi(1-x)-\psi(x)]~dx \right)$}\ .
$$
Using Taylor expansions of $v(x_i\pm hx)$ about $x_i$, we obtain
\begin{equation}\label{eq:basicdiff}
v(x_i+hx)-v(x_i-hx)=2hx v_i'+\frac{2h^3}{3!}v_i^{(3)} x^3+\cdots+\frac{2h^k}{k!} v_i^{(k)}x^k+h^{k+1}R_{i,k}(x;h)\ ,
\end{equation}
where $|R_{i,k}(x;h)|\leq C h \|v^{(k+2)}\|_\infty$ for some positive constant $C$ independent of $h$ and $v$, depending on $k$.

From the properties of $\psi$, we have
$$\int_0^1 x^j\left[(k+1)\psi(1-x)-\psi(x) \right]~dx=0,\quad j=1,2,\dots,k\ ,$$
which implies
$|\gamma_i|\leq C_k h^{k+3}\|v^{(k+2)}\|_\infty$.
Hence,
$$|(\bA\bbeta)_i|=\frac{h}{k(k+1)(k+2)}|\gamma_i|\leq \widetilde{C}_k h^{k+4}\|v^{(k+2)}\|_\infty\ ,$$
where $C_k$, $\widetilde{C}_k$ are positive constants that depend on $k$ but are independent of $h$ and $v$. By the equivalence (\ref{eq:twoptwo}), it follows that there exists a positive constant $C$, independent of $h$, such that $|\bbeta|_\infty \leq C h^{k+3} \|v^{(k+2)}\|_\infty$, which establishes the estimate (\ref{eq:superapprox}).
\end{proof}

\subsection{The dichotomy property}

Now we are ready to prove the main results of this section. 
\begin{theorem}\label{thm:main1}
\begin{enumerate}
\item[(i)] Let $k=2\kappa+1$, $\kappa=0,1,2,\dots$, and $u\in C_p^{k+2}$. Then there exists a constant $C_k>0$ independent of $h$ and $u$ such that
\begin{equation}\label{eq:onepone}
\|P[(Pu-u)_x]\|\leq C_k h^{k+1}\|u^{(k+2)}\|_\infty\ .
\end{equation}
\item[(ii)]
If $k=1$ and $u\in C_p^5$, then there exists a constant $C_1>0$, independent of $h$ and $u$, such that
\begin{equation}\label{eq:oneptwo}
\|P[(Pu-u)_x]\|\leq C_1 h^4\|u^{(5)}\|_\infty\ .
\end{equation}
\end{enumerate}
\end{theorem}
\begin{proof}
First, we prove (i). It’s worth noting that when $\kappa=0$, the proof becomes considerably simpler because $\mathcal{S}_h^2$ is reduced to just $\{0\}$, and the computations related to elements of $S_h^2$ are no longer required. However, to avoid repetition, we present a single proof that includes $k=1$. We define the error $e=Pu-u$. Using the decomposition of $P$, we have
$$Pu=\pi_1u+\pi_2u = \sum_{i=1}^N d_i \ell_i+\sum_{i=1}^N\sum_{m=1}^{k-1} c_{i,m}q_{i,m}\ ,$$
and
$$P(e_x)=\pi_1(e_x)+\pi_2(e_x)=\sum_{i=1}^N\delta_i\ell_i+\sum_{i=1}^N\sum_{m=1}^{k-1}\gamma_{m,i}q_{m,i}\ .$$
From the definition of the $L^2$-projection, we have
$$\|P(e_x)\|^2=(e_x,P(e_x))=\sum_{i=1}^N(e_x,\ell_i)\delta_i+\sum_{i=1}^N\sum_{m=1}^{k-1}(e_x,q_{i,m})\gamma_{i,m}\ .$$
To establish (\ref{eq:onepone}), it suffices to show the existence of a positive constant $C$, independent of $h$ and $u$, such that
$$|(e_x,\ell_i)|+|(e_x,q_{i,m})|\leq C h^{k+2}\|u^{(k+2)}\|_\infty \quad \text{for all}\quad 1\leq i\leq N, ~ 1\leq m\leq k-1\ .$$
This is due to (\ref{eq:estpi1}) and (\ref{eq:estpi2}) of Lemma \ref{lem:estimatepis}.

To estimate $(e_x,\ell_i)$, we first decompose it using the projections $\pi_1$ and $\pi_2$:
\begin{equation}\label{eq:threepone}
(e_x,\ell_i)=((\pi_1u)_x,\ell_i)+((\pi_2u)_x-u_x,\ell_i)\ .
\end{equation}
For the first term, using property (\ref{eq:viii}) of Lemma \ref{lem:usefl} for $\psi$, we obtain
\begin{equation}\label{eq:threeptwo}
((\pi_1u)_x,\ell_i)=d_{i-1}(\ell_{i-1}',\ell_i)+d_{i+1}(\ell_{i+1}',\ell_i)=\frac{1}{k+1}(d_{i+1}-d_{i-1})\ .
\end{equation}
For the second term, integration by parts gives
$$((\pi_2u)_x-u_x,\ell_i)=-(\pi_2 u-u,\ell_i')\ .$$

We now consider the function $p_i\in\mathcal{S}_h^2$ defined by
$$p_i(x)=\left\{\begin{array}{ll}
\ell_i'(x)-\ell_i'(x_i^-)\ell_i(x)-\ell_i'(x_{i-1}^+)\ell_{i-1}(x), & x\in I_i\ ,\\
\ell_i'(x)-\ell_i'(x_i^+)\ell_i(x)-\ell_i'(x_{i+1}^-)\ell_{i+1}(x), & x\in I_{i+1}\ ,\\
0, & \text{otherwise},
\end{array}
\right. $$
where $I_i=(x_{i-1},x_i)$.
Since $p_i$ is continuous and satisfies $p_i(x_i)=0$, by definition of $\pi_2$, and the orthogonality of $\mathcal{S}_h^1$ and $\mathcal{S}_h^2$, we have
$$
\begin{aligned}
(\pi_2u-u,\ell_i')&= (\pi_2u-u,\ell_i')_{I_i}+(\pi_2u-u,\ell_i')_{I_{i+1}}\\
&= (\pi_2u-u,p_i)_{I_i}+\ell_i'(x_i^-)(\pi_2u-u,\ell_i)_{I_{i}}+\ell_i'(x_{i-1}^+)(\pi_2u-u,\ell_{i-1})_{I_i}\\
&\quad +(\pi_2u-u,p_i)_{I_{i+1}}+\ell_i'(x_i^+)(\pi_2u-u,\ell_i)_{I_{i+1}}+\ell_i'(x_{i+1}^-)(\pi_2u-u,\ell_{i+1})_{I_{i+1}}\\
&=-\ell_i'(x_i^-)(u,\ell_i)_{I_i}-\ell_i'(x_{i-1}^+)(u,\ell_{i-1})_{I_i}-\ell_i'(x_i^+)(u,\ell_i)_{I_{i+1}}-\ell_i'(x_{i+1}^-)(u,\ell_{i+1})_{I_{i+1}}\ , 
\end{aligned} 
$$
where $\ell_i'(x_j^{\pm})$ for $j=i-1,i,i+1$ stands for the left and right limits $\ell_i'(x_j^{\pm}
)=\lim_{x\to x_j^{\pm}}\ell_i'(x)$.
This can be further simplified to
\begin{equation}\label{eq:innerprod2}
\text{ \footnotesize 
$\begin{aligned}
(\pi_2u-u,\ell_i')&=   -\frac{1}{h}\psi'(1)(u,\ell_i)_{I_i}-\frac{1}{h}\psi'(0)(u,\ell_{i-1})_{I_i}+\frac{1}{h}\psi'(1)(u,\ell_i)_{I_{i+1}}+\frac{1}{h}\psi'(0)(u,\ell_{i+1})_{I_{i+1}}\\
&= \frac{1}{h}\psi'(1)\left[(u,\ell_i)_{I_{i+1}}-(u,\ell_i)_{I_i}\right]+\frac{1}{h}\psi'(0)\left[(u,\ell_{i+1})_{I_{i+1}}-(u,\ell_{i-1})_{I_i}\right]\ .
\end{aligned}$} 
\end{equation}
On the other hand,
$$(u,\ell_i)_{I_{i+1}}=\int_{x_i}^{x_{i+1}}u(x)\psi\left(\frac{x_{i+1}-x}{h} \right)~dx=h\int_0^1 u(x_i+hx)\psi(1-x)~dx\ ,$$
$$(u,\ell_i)_{I_i}=\int_{x_{i-1}}^{x_i}u(x)\psi\left(\frac{x-x_{i-1}}{h} \right)~dx=h\int_0^1 u(x_i-hx)\psi(1-x)~dx\ ,$$
and
$$(u,\ell_{i+1})_{I_{i+1}}=h\int_0^1 u(x_i+hx)\psi(x)~dx,\quad (u,\ell_{i-1})_{I_i}=h\int_0^1 u(x_i-hx)\psi(x)~dx\ .$$
Substituting these inner products into (\ref{eq:innerprod2}), leads to
\begin{equation}\label{eq:innerprod3}
(\pi_2u-u,\ell_i')=\int_0^1 \left[u(x_i+hx)-u(x_i-hx) \right]\left[\psi'(1)\psi(1-x)+\psi'(0)\psi(x) \right]~dx\ .
\end{equation}
Note that from the properties (\ref{eq:vi}), (\ref{eq:iv}) and (\ref{eq:vii}) of Lemma \ref{lem:usefl} for $\psi$ we have that
$$\int_0^1 x^j\left[\psi'(1)\psi(1-x)+\psi'(0)\psi(x)\right]~dx=\frac{1}{k+1},\quad j=1,2,\dots,k\ .$$

Using Taylor expansions of $u(x_i+hx)$ and $u(x_i-hx)$ about $x_i$, we obtain
$$(\pi_2u-u,\ell_i')=\frac{1}{k+1}\left(2hu_i'+\frac{2h^3}{3!}u_i'''+\cdots+\frac{2h^k}{k!} u_i^{(k)}\right)+h^{k+1}R_{i,k}(h)\ ,$$
where $|R_{i,k}(h)|\leq C_k h \|u^{(k+2)}\|_\infty$, with $C_k>0$ independent of $h$ and $u$. 
Similarly, using Taylor expansions for $u(x_i+h)$ and $u(x_i-h)$, we have 
$$(\pi_2u-u,\ell_i')=\frac{1}{k+1}\left(u(x_i+h)-u(x_i-h) \right)+h^{k+1}\widetilde{R}_{i,k}(h)\ ,$$
where $|\widetilde{R}_{i,k}(h)|\leq \widetilde{C}_k h \|u^{(k+2)}\|_\infty$ for some constant $\widetilde{C}_k>0$ independent of $h$ and $u$.
Combining this with (\ref{eq:threepone}) and (\ref{eq:threeptwo}) gives
\begin{equation}\label{eq:intermid}
(e_x,\ell_i)=\frac{1}{k+1}\left((d_{i+1}-u_{i+1})-(d_{i-1}-u_{i-1}) \right)-h^{k+1} \widetilde{R}_{i,k}(h)\ .
\end{equation}

Applying Lemma \ref{lem:errestpi1} with $v=v(x)=u(x+h)-u(x-h)$ and using (\ref{eq:intermid}), we then obtain
$$|(e_x,\ell_i)|\leq Ch ^{k+2}\|u^{(k+2)}\|_\infty\ .$$
Note that due to periodicity $(\pi_1v)(x)=(\pi_1u)(x+h)-(\pi_1u)(x-h)$.

For $(e_x,q_{i,m})$, we have
\begin{equation}\label{eq:qfirst}
(e_x,q_{i,m})=-(e,q_{i,m}')=-(\pi_1 u+\pi_2u -u,q_{i,m}')=-(\pi_1u,q_{i,m}')-(\pi_2u-u,q_{i,m}')\ .
\end{equation}
It suffices to consider $m=1$ and $m=k-1$, since each of the two inner products on the right-hand side of (\ref{eq:qfirst}) vanishes for $m=2,3,\dots,k-2$ and $k>3$ (see Remark \ref{rmk:qim}). Note that if $k=3$, then $m=1$ or $m=k-1=2$, and such cases are analyzed below.

Case $m=1$: Let
$p_i(x)=q_{i,1}'(x)-q_{i,1}'(x_{i-1}^+)\ell_{i-1}(x)$,  $x\in I_i$
so that $p_i(x_{i-1})=p_i(x_i)=0$. Then
$$\begin{aligned}
-(\pi_1u,q_{i,1}')-(\pi_2u-u,q_{i,1}') &=-q_{i,1}'(x_{i-1}^+)\left[(\pi_1 u,\ell_{i-1})_{I_i}+(\pi_2u-u,\ell_{i-1})_{I_i} \right]\\
&=-q_{i,1}'(x_{i-1}^+)(\pi_1u-u,\ell_{i-1})_{I_i}\ ,
\end{aligned}$$
which implies
$(e_x,q_{i,1})=q_{i,1}'(x_{i-1}^+)(\pi_1 u-u,\ell_{i-1})_{I_{i}}$. 
Note again that in this notation 
$q_{i,\nu}'(x_j^{\pm})=\lim_{x\to x_j^{\pm}}q_{i,\nu}'(x)$.
Using (\ref{eq:superapprox}) of Lemma \ref{lem:superapprox} and the fact that $q_{i,1}'(x_{i-1}^+)=1/h$, we obtain
$$|(e_x,q_{i,1})|\leq C_1 h^{k+2}\|u^{(k+2)}\|_\infty\ ,$$
for some constant $C_1>0$, independent of $h$ and $u$.

Case $m=k-1$: Similarly, we have 
$(e_x,q_{i,k-1})=-q_{i,k-1}'(x_i^-)(\pi_1 u-u,\ell_i)_{I_i}$.
Thus, from (\ref{eq:superapprox}) of Lemma \ref{lem:superapprox} and the fact that $q_{i,k-1}'(x_{i}^-)=-1/h$, we obtain
$$|(e_x,q_{i,k-1})|\leq C_2 h^{k+2}\|u^{(k+2)}\|_\infty\ ,$$
for some constant $C_2>0$, independent of $h$ and $u$, and the proof of (i) is complete.

For the case (ii), recall that for $k=1$, we have $\mathcal{S}_h^2=\{0\}$. Let $e=Pu-u$. Then $Pu=\sum_{i=1}^N d_i\ell_i$, and $P(e_x)=\sum_{i=1}^N\delta_i\ell_i$. Moreover,
$$\|P(e_x)\|^2=(e_x,P(e_x))=\sum_{i=1}^N \zeta_i\delta_i\ ,$$
where $\zeta_i=(e_x,\ell_i)$. In particular, similarly to (\ref{eq:threeptwo}), 
$$\zeta_i=((Pu)_x-u_x,\ell_i)=\frac{1}{2}(d_{i+1}-d_{i-1})-(u_x,\ell_i)\ .$$
Let $\bzeta$ denote the vector with entries $\zeta_i$. Then,
$$(\bA\bzeta)_i=\frac{1}{2}(u,\ell_{i+1}-\ell_{i-1})-\frac{h}{6}(u_x,4\ell_i+\ell_{i+1}+\ell_{i-1})=(\tfrac{1}{2}v-\tfrac{h}{6}w,\ell_i)\ ,$$
where $v(x)=u(x+h)-u(x-h)$ and $w(x)=4u_x(x)+u_x(x+h)+u_x(x-h)$. 
Using Taylor expansions of $u$ and $u_x$ about $x$, we can write
$$
v(x) =2hu_x(x)+\frac{h^3}{3}u^{(3)}(x)+h^5 r(x),\qquad 
w(x) = 6 u_x(x)+h^2 u^{(3)}(x)+h^4 \widetilde{r}(x)\ ,$$
where $|r(x)|+|\widetilde{r}(x)|\leq c\|u^{(5)}\|_\infty$, for some constant $c>0$ independent of $h$ and $u$. 
Consequently, $|(\bA\bzeta)_i|\leq \tilde{c} h^6\|u^{(5)}\|_\infty$,
for some constant $\tilde{c}>0$, independent of $h$ and $u$. The last estimate together with (\ref{eq:twoptwo}) and (\ref{eq:twopone}) then yields the desired estimate (\ref{eq:oneptwo}).
\end{proof}
Repeating the preceding construction for even values of $k$, one obtains that the estimate (\ref{eq:standard}) cannot be improved.
%
%
\begin{remark}
The previous results hold exactly the same if one assumes that $v\in W_p^{r,\infty}$ instead of $v\in C_p^{r}$.
\end{remark}

\subsection{Other approximation properties of $\mathcal{V}_h^k$}

We close this section by reviewing some approximation properties and an inverse inequality that will be used in the error estimates of the subsequent sections.

Let $k\geq 1$ be an integer, and let $\cup_{i=1}^{N}\{q_{i,1},q_{i,2},\dots, q_{i,k-1}\}$ be a basis of $\mathcal{S}_h^2$. We define the operator $I_h:C_p\to \mathcal{V}_h^k$ by 
$$I_hv=\sum_{i=1}^N\left[v(x_i)\ell_i+\sum_{m=1}^{k-1} c_{i,m} q_{i,m} \right]\ ,$$
where the coefficients $c_{i,1},c_{i,2},\dots,c_{i,k-1}$ are chosen such that, on each interval $\bar{I}_i=[x_{i-1},x_i]$, the function $I_h v$ is the interpolating polynomial of $v$ at the points $x_{i-1}+j\cdot h/k$, $j=0,1,\dots,k$. 
For the interpolation error we have the following estimate, and we provide an elegant proof for the sake of completeness:
\begin{proposition}
Let $u\in H_p^{k+1}$. Then
\begin{equation}\label{eq:appropro}
    \|u-I_h u\|+h\|(u-I_hu)_x\|\leq h^{k+1}\|u^{(k+1)}\|\ .
\end{equation}
\end{proposition}
\begin{proof}
Let $r_h(x)=u(x)-(I_hu)(x)$. Then $r_h$ has at least $k+1$ distinct roots in each interval $\bar{I}_i$. By applying Rolle's Theorem inductively, we obtain that $r_h^{(\nu)}$, for $\nu=1,2,\dots,k$ has $k+1-\nu$ distinct roots in $I_i$. In particular, for $\nu=k$, if $\xi_i$ is a root of $r_h^{(k)}$ in $I_i$, then for all $x\in\bar{I}_i$,
 $r_h^{(k)}(x)=\int_{\xi_i}^x r_h^{(k+1)}(y)~dy=\int_{\xi_i}^x u^{(k+1)}(y)~dy\ .$
By the Cauchy-Schwarz inequality,
$[r_h^{(k)}(x)]^2\leq h\int_{x_{i-1}}^{x_i}[u^{(k+1)}(y)]^2~dy\ .$ 
Integrating over $I_i$ gives $\|r_h^{(k)}\|^2_{L^2(I_i)}\leq h^2\|u^{(k+1)}\|^2_{L^2(I_i)}$.
Summing over $i=1,2,\dots, N$, yields $\|r_h^{(k)}\|\leq h\|u^{(k+1)}\|$.
Applying the same argument recursively, we obtain $\|r_h^{(n-1)}\|\leq h\|r_h^{(n)}\|$, for $n=k,k-1,\dots,1$, and the desired estimate follows.
\end{proof}
It is worth noting that there exists a constant $C$, depending on $k$ but independent of $h$, such that, for any $\chi\in\mathcal{V}_h^k$, 
\begin{equation}\label{eq:inverse}
\|\chi'\|\leq Ch^{-1}\|\chi\|\ .
\end{equation}
This is the standard inverse inequality for one-dimensional Lagrange finite element spaces, and its proof follows the same lines as in the classical references; cf., e.g., \cite{EG2004}.
Finally, using the previous estimates for the interpolant and the inverse inequality, we obtain the classical error estimate $\|u-Pu\|+h\|(u-Pu)_x\|\leq C h^{k+1}\|u^{(k+1)}\|$ for any $u\in H_p^{k+1}$. For more information, we refer to \cite{EG2004}.

\subsection{Experimental validation}

We experimentally validate the Theorem \ref{thm:main1} using the function $u(x)=\sin(2\pi x)$ on the interval $[0,1]$. The interval was discretized into $N(m)$  subintervals of length $h = 1/N(m)$ for $m=1,2,\dots,5$, and the error $E_m^k=\|P[(Pu-u)_x]\|$ was computed for $k=1,2,\dots, 7$ using the FEniCS library of Python \cite{Fenics}. The experimental convergence rate $r_m$ approximating the exponent $r$ of the inequality $\|P[(Pu-u)x]\|\leq Ch^r$ was estimated such that 
$$r_m=\frac{\log_{10}(E_{m-1}^k/E_m^k)}{\log_{10}(N(m)/N(m-1))}\ .$$

The results presented in Table \ref{tab:exval1} confirm the theoretical predictions. In particular, they show that the convergence rate is $4$ for piecewise linear elements ($k=1$), $k+1$  for odd values of $k>1$, and $k$ for even values of $k>1$. Different values of $N(m)$ were used for different values of $k$ because, as $k$ increases, the errors become extremely small and are affected by finite-precision effects, which make the estimation of the convergence rate inaccurate. For this reason, the experiments are also limited to $k \le 7$.

\begin{table}[ht!]\label{tab:exval1}
\begin{tabular}{lllll|lll|ll}
\hline
$h$ & $k=1$ & $k=2$ & $k=3$ & $k=4$ & $h$ & $k=5$ & $k=6$ & $h$ & $k=7$ \\ \hline
0.1 & -- & -- & -- & -- & 0.2 & -- & -- & 0.5 & -- \\
0.05 & 4.051 & 1.755 & 4.147 & 3.844 & 0.1 & 6.243 & 5.526 & 0.25 & 8.442 \\ 
0.02 & 4.011 & 1.946 & 4.034 & 3.966 & 0.066 & 6.095 & 5.842 & 0.125 & 8.206 \\
0.01 & 4.002 & 1.990 & 4.007 & 3.994 & 0.05 & 6.049 & 5.921 & 0.1 & 8.105 \\ 
0.005 & 4.001 & 1.997 & 4.002 & 3.998 & 0.04 & 6.030 & 5.953 & 0.05 & 8.040 \\\hline
theory & 4 & 2 & 4 & 4 & theory & 6 & 6 & theory & 8  \\ \hline
\end{tabular}
\caption{Convergence rates of the error $E_m^k=\|P[(Pu-u)_x]\|$ (periodic boundary conditions)}
\end{table}

\subsection{Other approximation properties of $\mathcal{V}_h^k$}

We close this section by reviewing some approximation properties and an inverse inequality that will be used in the error estimates of the subsequent sections.

Let $k\geq 1$ be an integer, and let $\cup_{i=1}^{N}\{q_{i,1},q_{i,2},\dots, q_{i,k-1}\}$ be a basis of $\mathcal{S}_h^2$. We define the operator $I_h:C_p\to \mathcal{V}_h^k$ by 
$$I_hv=\sum_{i=1}^N\left[v(x_i)\ell_i+\sum_{m=1}^{k-1} c_{i,m} q_{i,m} \right]\ ,$$
where the coefficients $c_{i,1},c_{i,2},\dots,c_{i,k-1}$ are chosen such that, on each interval $\bar{I}_i=[x_{i-1},x_i]$, the function $I_h v$ is the interpolating polynomial of $v$ at the points $x_{i-1}+j\cdot h/k$, $j=0,1,\dots,k$. For the interpolation error we have the following estimate, and we provide an elegant proof for the sake of completeness:
\begin{proposition}
Let $u\in H_p^{k+1}$. Then
\begin{equation}\label{eq:appropro}
    \|u-I_h u\|+h\|(u-I_hu)_x\|\leq h^{k+1}\|u^{(k+1)}\|\ .
\end{equation}
\end{proposition}
\begin{proof}
Let $r_h(x)=u(x)-(I_hu)(x)$. Then $r_h$ has at least $k+1$ distinct roots in each interval $\bar{I}_i$. By applying Rolle's Theorem inductively, we obtain that $r_h^{(\nu)}$, for $\nu=1,2,\dots,k$ has $k+1-\nu$ distinct roots in $I_i$. In particular, for $\nu=k$, if $\xi_i$ is a root of $r_h^{(k)}$ in $I_i$, then for all $x\in\bar{I}_i$,
 $r_h^{(k)}(x)=\int_{\xi_i}^x r_h^{(k+1)}(y)~dy=\int_{\xi_i}^x u^{(k+1)}(y)~dy\ .$
By Cauchy-Schwarz inequality,
$[r_h^{(k)}(x)]^2\leq h\int_{x_{i-1}}^{x_i}[u^{(k+1)}(y)]^2~dy\ .$ 
Integrating over $I_i$ gives $\|r_h^{(k)}\|^2_{L^2(I_i)}\leq h^2\|u^{(k+1)}\|^2$.
Summing over $i=1,2,\dots, N$, yields $\|r_h^{(k)}\|\leq h\|u^{(k+1)}\|$.
Applying the same argument recursively, we obtain $\|r_h^{(n-1)}\|\leq h\|r_h^{(n)}\|$, for $n=k,k-1,\dots,1$, and the desired estimate follows.
\end{proof}
It is worth noting that, for any $\chi\in\mathcal{V}_h^k$, there exists a constant $C$, depending on $k$ and independent of $h$, such that 
\begin{equation}\label{eq:inverse}
\|\chi'\|\leq Ch^{-1}\|\chi\|\ .
\end{equation}
This is the standard inverse inequality for one-dimensional Lagrange finite element spaces, and its proof follows the same lines as in the classical references cf. eg. \cite{EG2004}.
Finally, using the previous estimates for the interpolant and the inverse inequality, we obtain the classical error estimate $\|u-Pu\|+h\|(u-Pu)_x\|\leq C h^{k+1}\|u^{(k+1)}\|$ for any $u\in H_p^{k+1}$. For more information, we refer to \cite{EG2004}.

\section{Application to a conservative finite element method}\label{sec:galerkin}

We consider the initial-periodic boundary value problem for the RLW equation \cite{BD1980}:
\begin{equation}\label{eq:ipbvp}
\begin{aligned}
&u_t+u_x+uu_x-u_{xxt}= 0 \quad\text{for $x\in (a,b)$ and $t>0$}\ ,\\
& \partial^i_x u(a,t)=\partial^i_x u(b,t) \quad \text{ for $t\geq 0$ and $i=0,1,\dots$}\ ,\\
& u(x,0)=U_0(x)\quad \text{for $x\in[a,b]$}\ .
\end{aligned}
\end{equation}
This problem effectively mimics the Cauchy problem on $\mathbb{R}$ and is therefore particularly suitable for studying long-time propagation phenomena, including the interaction of solitary waves and the resolution properties of arbitrary initial conditions.

Any sufficiently smooth solution $u=u(x,t)$ of (\ref{eq:ipbvp}) satisfies the following conservation laws:
\begin{align}
&\frac{d}{dt}\mathcal{M}(t;u)=0, \quad\text{where}\quad\mathcal{M}(t;u)=\int_a^b u~dx\qquad \text{(Mass)} \ , \label{eq:massc} \\
&\frac{d}{dt}\mathcal{I}(t;u)=0,\quad \text{where} \quad \mathcal{I}(t;u)=\frac{1}{2}\int_a^b u^2+u_x^2~dx\qquad \text{(Impulse)}\ . \label{eq:momentumc} \\
&\frac{d}{dt}\mathcal{E}(t;u)=0,\quad \text{where} \quad \mathcal{E}(t;u)=\frac{1}{2}\int_a^b u^2+\frac{1}{3}u^3~dx\qquad \text{(Energy)}\ . \label{eq:energyc} 
\end{align}
Note that these are the only independent conservation laws of (\ref{eq:ipbvp}), cf. \cite{Olver1979}. 
It is apparent that the RLW equation is a Hamiltonian partial differential equation, with the Hamiltonian functional given by $\mathcal{E}$. In this framework, it can be expressed in the canonical form 
\begin{equation}\label{eq:hamilton}
u_t=-\mathcal{D}\frac{\delta\mathcal{E}}{\delta u}\ ,   
\end{equation}
where $\mathcal{D}=(I-\partial_{xx})^{-1}\partial_x$ is a skew-adjoint operator, and $\delta\mathcal{E}/\delta u$ denotes the Gateaux derivative of the Hamiltonian $\mathcal{E}$.

\subsection{A conservative finite element semi-discretization}

While the following analysis also holds on any interval $(a,b)$, we focus again on the simpler case $\bar{I}=[a,b]=[0,1]$, and use a uniform grid as in the preceding sections. 

A modified Galerkin/finite element semi-discretization of (\ref{eq:ipbvp}) is defined as follows: Determine $\tu:[0,T]\to \mathcal{V}_h^k$ for $k=1,2,\dots$,  such that 
\begin{equation}\label{eq:semid}
(\tu_t,\chi)+(P[\tu_{xt}],\chi_x)=(P[\tu+\tfrac{1}{2}\tu^2],\chi_x)\quad \text{for all $\chi\in \mathcal{V}_h^k$}\ ,
\end{equation}
with initial condition $\tilde{u}(x,0)=P[U_0(x)]$ and $T>0$ the maximal time of existence of the analytical solution $u$ of (\ref{eq:ipbvp}). This semi-discretization is energy-conserva\-tive, and we have the following:
\begin{proposition}\label{prop:conserv}
 Any solution $\tu(\cdot,t)\in \mathcal{V}_h^k$ of the semi-discretization (\ref{eq:semid}) of the initial-periodic boundary value problem (\ref{eq:ipbvp}) conserves the mass and energy functionals in the sense that
 $$\frac{d}{dt}\mathcal{M}(t;\tu)=0\quad \text{and}\quad \frac{d}{dt}\mathcal{E}(t;\tu)=0\ .$$
\end{proposition}
\begin{proof}
The mass conservation follows directly from (\ref{eq:semid}) by choosing $\chi=1$.
To prove energy conservation, rewrite (\ref{eq:semid}) as
\begin{equation}\label{eq:first}
(\tu_t,\chi)=(P[\tu+\tfrac{1}{2}\tu^2-\tu_{xt}],\chi_x)\quad \text{for all $\chi\in S_h$}\ . 
\end{equation}
Denote $R=\tu+\tfrac{1}{2}\tu^2-\tu_{xt}$ and set $\chi=P[R]$ in (\ref{eq:first}). Due to the periodic boundary conditions, we then have
\begin{equation}\label{eq:zero}
(\tu_t,R) = (\tu_t,P[R])=(P[R],P[R]_x) =0\ .
\end{equation}
On the other hand,
$$
\begin{aligned}
(\tu_t,R) &= (\tu_t,\tu+\tfrac{1}{2} \tu^2-\tu_{xt})=\tfrac{1}{2}\int_0^1(\tu^2+\tfrac{1}{3}\tu^3)_t-(\tu_t^2)_x~dx=\tfrac{1}{2}\frac{d}{dt}\int_0^1 \tu^2+\tfrac{1}{3}\tu^3~dx=\frac{d}{dt}\mathcal{E}(t;\tu)\ .
\end{aligned}
$$
Combining this with (\ref{eq:zero}) yields the desired energy conservation.
\end{proof}

The semi-discretization (\ref{eq:semid}) can also be expressed in the form:
\begin{alignat}{2}
(\tu_t,\chi)+(\tw_t,\chi_x)&=(\tz,\chi_x) & & \quad \text{for all $\chi\in \mathcal{V}_h^k$}\ ,\label{eq:mixfem1}\\
(\tu_{xt},\psi)-(\tw_t,\psi) &= 0 & & \quad \text{for all $\psi\in \mathcal{V}_h^k$}\ ,\label{eq:mixfem2}\\
(\tz,\phi)&=(\tu+\tfrac{1}{2}\tu^2,\phi)& & \quad \text{for all $\phi\in \mathcal{V}_h^k$}\ .\label{eq:mixfem3}
\end{alignat}
In analogy, we can write the RLW equation as
\begin{align}
& u_t-w_{xt}=-z_x,\label{eq:bbmn1}\\
& w = u_x, \label{eq:bbmn2}\\
& z = u+\tfrac{1}{2}u^2\ . \label{eq:bbmn3}
\end{align}
Note that if $u(x,0)=U_0(x)$, then naturally $w(x,0)=W_0(x)=U_0'(x)$. 
We therefore consider as initial conditions of the semi-discrete system (\ref{eq:mixfem1})--(\ref{eq:mixfem3}) the projections
\begin{equation}\label{eq:ics}
\tilde{u}(x,0)=PU_0(x)\quad \text{and}\quad  \tilde{w}(x,0)=PW_0(x)=P[U_0'](x)\ ,
\end{equation}
for all $x\in[0,1]$.

\begin{remark}\label{rmk:doubleu}
From (\ref{eq:mixfem2}), we have $\tilde{w}_t={P[\tilde{u}_x]}_t$. Integrating in time over any interval $[0,t]$, gives  $\tilde{w}=P[\tilde{u}_x]+\tilde{w}(x,0)-P[U_0']=P[\tilde{u}_x]+PW_0-P[U_0']$. Using the initial condition (\ref{eq:ics}), we immediately obtain 
\begin{equation}\label{eq:doubleu}
\tilde{w}=P[\tilde{u}_x]\ .
\end{equation}
\end{remark}

\subsection{Proof of convergence}

In this section, we prove that the semi-discrete solution $\tilde{u}$ exists and converges to the classical solution with optimal order when $k\geq 1$ is an odd integer for all $t\in[0,T]$.

\begin{theorem}\label{thrm:Theorem1}
Let $u\in C([0,T];C^{k+2}_p)$ be the unique solution of the initial-periodic boundary value problem (\ref{eq:ipbvp}), and let $k=2\kappa+1$ for $\kappa=0,1,\dots$. Then, for sufficiently small $h>0$, there exists a unique solution $\tilde{u}(\cdot,t)\in \mathcal{V}_h^k$ of the semi-discrete problem (\ref{eq:semid}) (and equivalently of (\ref{eq:mixfem1})--(\ref{eq:mixfem3}))  for all $t\in[0,T]$, and a constant $C>0$ independent of $h$ such that
$$\|u-\tilde{u}\|\leq C h^{k+1}, \quad \|u_x-P[\tilde{u}_x]\|\leq C h^{k+1}\quad \text{and}\quad \|u_x-\tilde{u}_x\|\leq C h^k \ .$$
\end{theorem}
\begin{proof}
We define
$$\zeta=u-Pu,\quad \rho=Pu-\tilde{u},\quad \sigma=w-Pw,\quad \theta=Pw-\tilde{w}, \quad\tau=z-Pz,  \quad \xi=Pz-\tilde{z}\ .$$
Note that $\tilde{w}=P[\tilde{u}_x]$ and $\tilde{z}=P[\tilde{u}+\frac{1}{2}\tilde{u}^2]$. 
Throughout the analysis, we will use a generic positive constant $C$ independent of $h$.
We also assume that the exact solution is bounded, i.e., there is a constant $M$ such that $\|u\|+\|u_x\|\leq M$ for all $t\in[0,T]$. Then, for sufficiently small $h$, the semi-discrete solution satisfies $\|\tilde{u}(\cdot,0)\|+\|\tilde{w}(\cdot,0)\|\leq 2M$.

Upon choosing appropriate basis functions $\varphi_i(x)$, $i=1,2,\dots,kN$ for the finite element space $\mathcal{V}_h^k$, the semi-discrete solution can be represented as a linear combination of the basis functions:
$$\tilde{u}(x,t)=\sum_{i=1}^{kN} u_i(t) \varphi_i(x) \ \text{ and } \ \tilde{w}(x,t)=\sum_{i=1}^{kN} w_i(t)\varphi_i(x)\ .$$
Denote $\by$ the vector $\by=(\bu,\bw)^T$ where $\bu=(u_0,u_1,\dots)^T$ and $\bw=(w_0,w_1,\dots)^T$.
Then, the finite element system (\ref{eq:mixfem1})--(\ref{eq:mixfem3}) can be written as a system of nonlinear ordinary differential equations
\begin{equation}\label{eq:sysnlodes}
   \bM \by' = \bF(\by)\ ,
\end{equation}
where 
$$\bM = \begin{pmatrix}
    \bA & \bB\\
    \bB & \bA
\end{pmatrix}\ ,$$
with $\bA_{ij}=(\varphi_i,\varphi_j)$, $\bB_{ij}=(\varphi_i,\varphi_j')$, and
$\bF=(\bz,\bo)$ with $\bz_j=(\tz,\varphi_j')$. Here, $\bA$ is symmetric and positive definite, while $\bB$ is skew-symmetric. Moreover, for any vector $\bx=(\bx_1, \bx_2)^T$ we have
$$
 \bx^T\bM \bx = \bx_1^T\bA\bx_1+\bx_1^T\bB\bx_2+\bx_2^T\bB\bx_1+\bx_2^T\bA\bx_1> \bx_1^T\bB\bx_2+\bx_1^T(\bx_2^T\bB)^T= 0\ ,
$$
which shows that $\bM$ is positive definite and therefore invertible.

For the system $\by'=\bM^{-1}\bF(\by)$, the mapping $\bM^{-1}\bF$ is locally Lipschitz as well as $\bF$. Thus, there is a $t^\ast\leq T$ such that the system (\ref{eq:mixfem1})--(\ref{eq:mixfem3}) has a unique solution $(\tilde{u}, \tilde{w})$ satisfying
$\|\tilde{u}\|+\|\tilde{w}\|\leq \tilde{M}$ for all $t\in [0,t^\ast]$, for some constant $\tilde{M}>0$ independent of $h$.

We define $\omega_w=P[w_{xt}]-w_{xt}$, $\omega_z=P[z_x]-z_x$ and $\omega_u=P[u^2]-u^2$. Then, for $\chi\in\mathcal{V}_h^k$, we have 
$$
\begin{aligned}
    (\rho_t,\chi) &= (P[u_t],\chi)-(\tilde{u}_t,\chi)= (P[w_x]_t-P [z_x],\chi) + (\tilde{w}_t-\tilde{z},\chi_x)\\
 &=(P[w_{xt}]-w_{xt},\chi)+(w_{xt},\chi)+(\tilde{w}_t,\chi_x)-(P[z_x]-z_x,\chi)-(z_x,\chi)-(\tilde{z},\chi_x)\\
 &=(\omega_w,\chi)-(w_t-P[w_t],\chi_x)-(P[w_t]-\tilde{w}_t,\chi_x)-(\omega_z,\chi)+(z-Pz,\chi_x)+(Pz-\tilde{z},\chi_x)\\
 &=(P[\omega_w],\chi)+(P[{(\sigma_t)}_x],\chi)+(P[{(\theta_t)}_x],\chi)-(P[\omega_z],\chi)-(P[\tau_x],\chi)-(P[\xi_x],\chi)\ .
\end{aligned}
$$
Equivalently, in operator form, this gives
\begin{equation}\label{eq:rhot}
\rho_t=P[\omega_w+(\sigma_t)_x+(\theta_t)_x-\omega_z-\tau_x-\xi_x]\ .
\end{equation}

From (\ref{eq:mixfem2}) and (\ref{eq:bbmn2}), we have that for any $\psi\in\mathcal{V}_h^k$
$$
\begin{aligned}
(\theta,\psi)&=(Pw-\tilde{w},\psi) = (w-\tilde{w},\psi)= (u_x-P[\tilde{u}_x],\psi)
= -(u-\tilde{u},\psi_x)\\
&= -(u-Pu,\psi_x)-(Pu-\tilde{u},\psi_x)
=  -(\zeta,\psi_x)-(\rho,\psi_x)=  (P[\zeta_x],\psi)+(P[\rho_x],\psi)\ .
\end{aligned}$$
Hence, we obtain the identity
\begin{equation}\label{eq:thetaP}
\theta=P[\zeta_x+\rho_x]\ .
\end{equation}
Using the last formula, we have
$$
\|P[\rho_x]\| =\|\theta-P[\zeta_x]\|\leq \|\theta\|+\|P[\zeta_x]\|\ .
$$
According to Theorem \ref{thm:main1}, we have $\|P[\zeta_x]\|\leq Ch^{k+1}$. (If $k=1$ and $u\in C_p^5$, the exponent in the last estimate is $4$. However, this won’t improve the error estimate because, for instance, (\ref{eq:xiineq}) below). 
Using this estimate, we obtain
\begin{equation}\label{eq:Prhox}
\|P[\rho_x]\|\leq Ch^{k+1}+\|\theta\|\ .
\end{equation}

Let $\phi\in\mathcal{V}_h^k$. Then,
$$
\begin{aligned}
(\xi,\phi) &= (Pz-\tilde{z},\phi)= (Pu-\tilde{u},\phi)+\tfrac{1}{2}(P[u^2]-\tilde{u}^2,\phi)= (\rho,\phi)+\tfrac{1}{2}(P[u^2]-u^2,\phi)+\tfrac{1}{2}(u^2-\tilde{u}^2,\phi)\\
&= (\rho,\phi)+\tfrac{1}{2}(P[u^2]-u^2,\phi)+\tfrac{1}{2}((u+\tilde{u})(u-\tilde{u}),\phi)= (\rho,\phi)+\tfrac{1}{2}(\omega_u,\phi)+\tfrac{1}{2}((u+\tilde{u})(\zeta+\rho),\phi)\ .
\end{aligned}
$$
Taking $\phi=\xi$ in the previous relation yields
$$\|\xi\|^2\leq \|\rho\|\|\xi\|+C h^{k+1}\|\xi\|+\|u+\tilde{u}\|_\infty (C h^{k+1}+\|\rho\|)\|\xi\|\ ,$$
which can be simplified to
\begin{equation}\label{eq:xiineq}
\|\xi\|\leq C(h^{k+1}+\|\rho\|)\ .
\end{equation}
Now, we can proceed with the main estimate. From (\ref{eq:rhot}) and (\ref{eq:thetaP}) we have
$$
\begin{aligned}
(\rho_t,\rho)+(\theta_t,\theta) &= (P[\omega_w+(\sigma_t)_x+(\theta_t)_x-\omega_z-\tau_x-\xi_x],\rho)+
(\theta_t, P[\zeta_x+\rho_x]) \\
&=(\omega_w-\omega_z,\rho)+(P[(\sigma_t-\tau)_x],\rho)-(\theta_t,P[\rho_x])-(\xi_x,\rho)+(\theta_t,P[\zeta_x])+(\theta_t,P[\rho_x]) \\
&=(\omega_w-\omega_z,\rho)+(P[(\sigma_t-\tau)_x],\rho)-(\xi_x,\rho)+(\theta_t,P[\zeta_x]) \\
&=(\omega_w-\omega_z,\rho)+(P[(\sigma_t-\tau)_x],\rho)+(\xi,P[\rho_x])+\tfrac{d}{dt}(\theta,P[\zeta_x])-(\theta,P[(\zeta_t)_x])
\ .
\end{aligned}
$$
The last term in the previous inequality can be bounded using the Cauchy-Schwarz inequality: 
$$
(\theta,P[(\zeta_t)_x])\leq \frac{1}{8}\|\theta\|^2+2\|P[(\zeta_t)_x]\|^2\leq C(\|P[(\zeta_t)_x]\|^2+\|\theta\|^2)\ .
$$
By Theorem \ref{thm:main1}, $\|P[(\zeta_t)_x]\|\leq C h^{k+1}$. Therefore,
\begin{equation}\label{eq:estim1}
(\theta,P[(\zeta_t)_x])\leq C( h^{2(k+1)} +\|\theta\|^2) \  .
\end{equation}
Similarly, using the inequalities (\ref{eq:xiineq}) and (\ref{eq:Prhox}), we obtain
\begin{equation}\label{eq:estim2}
(\xi,P[\rho_x]) \leq C(h^{2(k+1)}+\|\xi\|^2+\|\theta\|^2)\leq C(h^{2(k+1)} +\|\rho\|^2+\|\theta\|^2)\ .
\end{equation}
Thus,
\begin{equation}
\frac{1}{2}\frac{d}{dt}(\|\rho\|^2+\|\theta\|^2) \leq C(h^{2(k+1)}+\|\rho\|^2+\|\theta\|^2)+\tfrac{d}{dt}(\theta,P[\zeta_x]) \ .   
\end{equation}
Integrating this relation yields 
\begin{equation}\label{eq:lrel}
\|\rho\|^2+\|\theta\|^2\leq C h^{2(k+1)} + \frac{1}{8}\|\theta\|^2+2\|P[\zeta_x]\|^2 + C\int_0^t (\|\rho\|^2+\|\theta^2\|)\ ,
\end{equation}
where, from now on, $C$ depends on $t^\ast$.
Taking into account that $\|P[\zeta_x]\|\leq C h^{k+1}$, we can simplify (\ref{eq:lrel}) to
\begin{equation}
\|\rho\|^2+\|\theta\|^2\leq Ch^{2(k+1)}+C\int_0^t (\|\rho\|^2+\|\theta^2\|)\ .
\end{equation}
Applying Gronwall's inequality, we obtain the optimal error estimate
\begin{equation}\label{eq:help2}
\|\rho\|^2+\|\theta\|^2\leq Ch^{2(k+1)}\ ,
\end{equation}
for all $t\in [0,t^\ast]$, which immediately gives
\begin{equation}
\|u-\tilde{u}\|^2+\|u_x-P[\tilde{u}_x]\|^2\leq Ch^{2(k+1)}\ .
\end{equation}
Finally, using (\ref{eq:help2}) and the inverse inequality (\ref{eq:inverse}), we have 
$$\|u_x-\tilde{u}_x\|\leq \|(u-Pu)_x\|+\|\rho_x\|\leq C(h^k+h^{-1}\|\rho\|)\leq Ch^k\ .$$
Choosing $h>0$ sufficiently small such that $\|\tilde{u}\|\leq \|u\|+Ch^{k+1}<2M$ for $0\leq t\leq t^\ast$, 
we ensure that the numerical solution $\tilde{u}$ will remain bounded in $L^2$ up to $t=t^\ast$. Similarly, since $\|u_x-P[\tilde{u}_x]\|\leq Ch^{k+1}$, we have 
$\|\tilde{w}\|\leq \|u_x\|+Ch^{k+1}<2M$ for $0\leq t\leq t^\ast$.
By standard extension arguments from the theory of ordinary differential equations, we conclude that the solution (and the error estimates) can be extended to the full interval $t\in[0,T]$. 
\end{proof}

\begin{remark}\label{rmk:Theorem1}
If $k=2\kappa$ for $\kappa=1,2,\dots$, we have that $\|P[\varepsilon_x]\|\leq C h^{k}$, where $\varepsilon = \zeta$ or $\zeta_t$. Therefore, modifying the proof of Theorem \ref{thrm:Theorem1} accordingly, we have that 
$\|u-\tilde{u}\|\leq C h^{k}$, $\|u_x-P[\tilde{u}_x]\|\leq C h^{k}$ and $\|u_x-\tilde{u}_x\|\leq C h^{k-1}$.
\end{remark}

\begin{remark}
It was shown in \cite{AD2013} that if $u(\cdot,t)\in C^{5}_p$, then in the space of periodic cubic splines, $\|P[(v-I_h v)_x]\|=O(h^{k+1})$, where $I_h$ is the cubic spline interpolant. Consequently, the proof of Theorem \ref{thrm:Theorem1} can be modified to remain valid if we consider periodic cubic splines with small changes.
\end{remark}

Building on the preceding analysis, we show that the conservation of the impulse functional $\mathcal{I}$ for any smooth solution $\tilde{u}$ of (\ref{eq:mixfem1})--(\ref{eq:mixfem3}) is approximated with high accuracy.
\begin{proposition}\label{prop:impulsec}
If $\tilde{u}(\cdot,t)\in \mathcal{V}_h^k$ with $k=1,2,\dots$ be a solution of the initial value problem (\ref{eq:mixfem1})--(\ref{eq:mixfem3}), (\ref{eq:ics}), and let $h>0$ sufficiently small. Then, there exists a constant $C>0$, independent of $h$, such that
\begin{equation}\label{eq:approxprop}
\max_{t\in[0,T]}|\mathcal{I}(t;\tilde{u})-\mathcal{I}(0;\tilde{u})|\leq Ch^r\ ,
\end{equation}
where $r=2k$ if $k=2\kappa+1$ with $\kappa=0,1,\dots$, and $r=2k-2$ if $k=2\kappa$ with $\kappa=1,2,\dots$.
\end{proposition}
\begin{proof} 
Taking $\chi=\tilde{u}$ and using (\ref{eq:doubleu}) in (\ref{eq:mixfem1}), gives
\begin{equation} \label{eq:almost1}
(\tilde{u}_t,\tilde{u})+({P[\tilde{u}_x]}_t,\tilde{u}_x)=(\tilde{z},\tilde{u}_x)=(P[\tilde{u}+\tfrac{1}{2}\tilde{u}^2],\tilde{u}_x)=\tfrac{1}{2}(P[\tilde{u}^2],\tilde{u}_x)=\tfrac{1}{2}(\tilde{u}^2,P[\tilde{u}_x])\ .
\end{equation}
Thus,
$$
(\tilde{u}_t,\tilde{u})+(\tilde{u}_{xt},\tilde{u}_x)+(P[\tilde{u}_{xt}]-\tilde{u}_{xt},\tilde{u}_x)=\tfrac{1}{2}(\tilde{u}^2,P[\tilde{u}_x]-\tilde{u}_x)\ .$$
Using the definition of the $L^2$-projection, we write the last relationship into the form 
$$\frac{d}{dt}\mathcal{I}(t;\tilde{u})-(P[\tilde{u}_{xt}]-\tilde{u}_{xt},P[\tilde{u}_x]-\tilde{u}_x)=\tfrac{1}{2}(\tilde{u}^2-P[\tilde{u}^2],P[\tilde{u}_x]-\tilde{u}_x)\ ,$$
which is equivalent to
\begin{equation}\label{eq:dtimpulse}
\frac{d}{dt}\mathcal{I}(t;\tilde{u})=\tfrac{1}{2}\frac{d}{dt}\|P[\tilde{u}_x]-\tilde{u}_x\|^2+\tfrac{1}{2}(\tilde{u}^2-P[\tilde{u}^2],P[\tilde{u}_x]-\tilde{u}_x)\ .
\end{equation}
Using Theorem \ref{thrm:Theorem1}, we have that $\|P[\tilde{u}_x]-\tilde{u}_x\|\leq \|P[\tilde{u}_x]-u_x\|+\|u_x-\tilde{u}_x\|\leq Ch^k$ if $k=2\kappa+1$ for $\kappa=0,1,\dots$ while $\|P[\tilde{u}_x]-\tilde{u}_x\|\leq Ch^{k-1}$ for $k=2\kappa$, $\kappa=1,2,\dots$. Therefore, after integration of (\ref{eq:dtimpulse}), and taking absolute values, we obtain
$$|\mathcal{I}(t,\tilde{u})-\mathcal{I}(0,\tilde{u})|\leq \tfrac{1}{2}\|P[\tilde{u}_x]-\tilde{u}_x\|^2+\tfrac{1}{4}T(\max_{t\in[0,T]}\|\tilde{u}^2-P[\tilde{u}^2]\|^2+\max_{t\in[0,T]}\|P[\tilde{u}_x]-\tilde{u}_x\|^2)\ ,$$
which completes the proof.
\end{proof}

Commenting on the standard Galerkin method applied to the RLW equation is known to be conservative in the sense that it preserves the impulse functional $\mathcal{I}$ but not the energy. In general, the standard Galerkin method has been used alongside the symplectic two-stage implicit Gauss--Legendre Runge--Kutta method, which is also known to preserve quadratic conservation laws due to its symplecticity. This combination has ensured conservative and highly accurate simulations for the KdV equation while studying blow-up phenomena \cite{BDKM1995}, as well as for the KdV--KdV system \cite{BDM2007i,BDM2008ii} in the study of generalized solitary waves. For the RLW equation, the standard Galerkin method has been studied theoretically in \cite{Wahlbin1974,ADT1981}, although its conservation properties were never the main focus.

\subsection{Experimental validation}

For the numerical validation of Theorem \ref{thrm:Theorem1} and Remark \ref{rmk:Theorem1}, we consider the periodic initial-value problem (\ref{eq:ipbvp}) with a suitable nonzero right-hand side so that the function $u(x,t) = e^t \sin(2\pi (x - 2t))$ on the interval $[0,1]$ is an exact solution of the modified problem for all $t \in [0,1]$. The corresponding semi-discrete formulation (\ref{eq:mixfem1})--(\ref{eq:mixfem3}) is then solved for $h = 1/N$ for various values of $N$  using the classical fourth-order Runge–Kutta method. The time step $\Delta t$ is chosen sufficiently small so that $(\Delta t)^4 \ll h^{k+1}$. The resulting errors, $\|u - \tilde{u}\|$, $\|u_x - P[\tilde{u}_x]\|$ and $\|u_x-\tilde{u}_x\|$, are recorded at $t = 1$ for $k = 1, 2, \dots, 6$.
\begin{figure}[t]
\centering
\includegraphics[width=0.7\columnwidth]{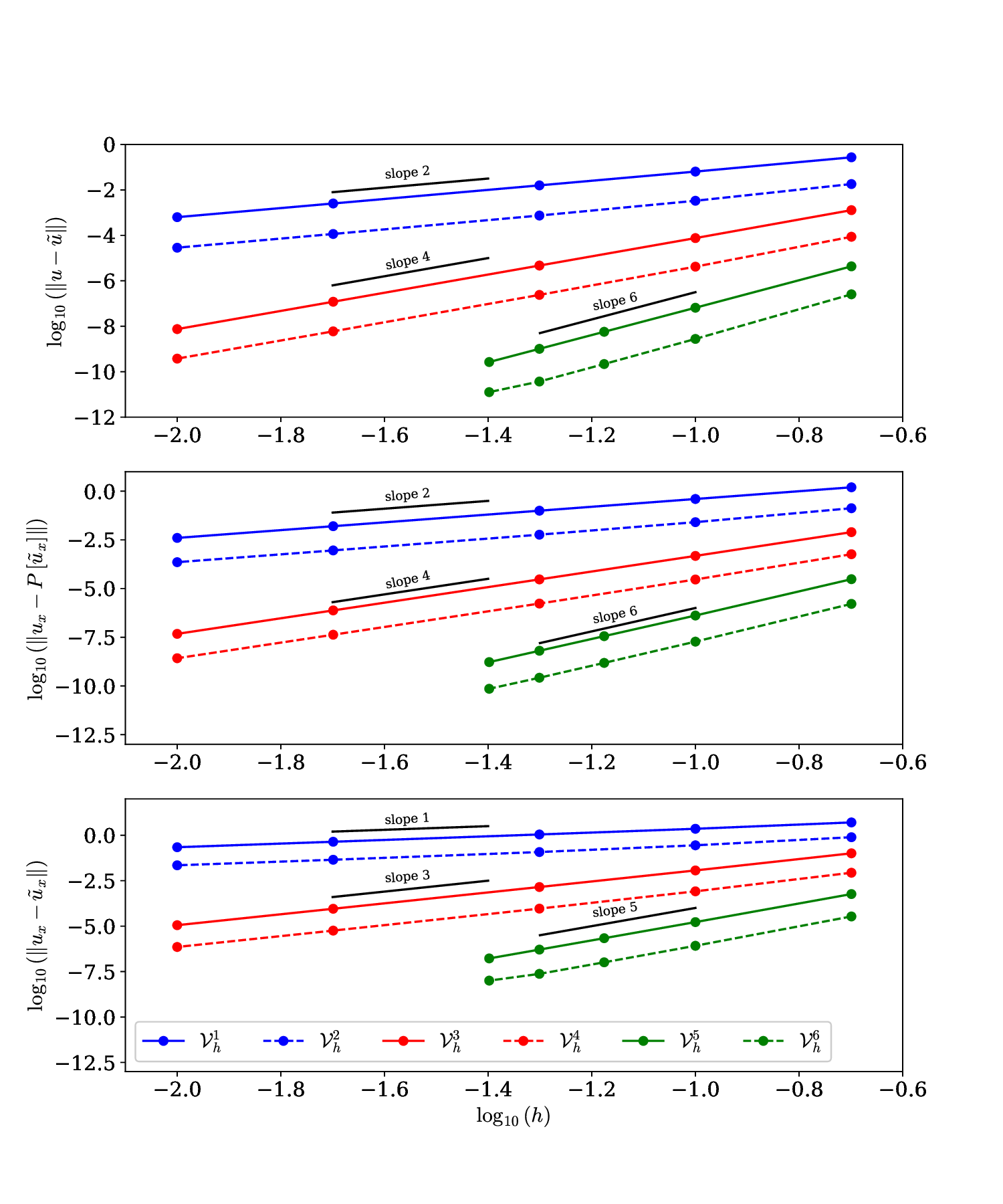}
\caption{Experimental convergence rates for $k=1,2,\dots,6$}
\label{fig:rates}
\end{figure}
Figure \ref{fig:rates} shows logarithmic plots of these errors versus the mesh size $h$ for $k = 1, 2, \dots, 6$. The data form straight lines with slopes close to $2$, $4$, and $6$, respectively, which confirms the expected convergence behavior. It is also noteworthy that, for even values of $k$, the convergence rates associated with the spaces $\mathcal{V}_h^{k-1}$ and $\mathcal{V}_h^k$ are nearly the same, yet the errors are consistently smaller in the latter case.

It is worth noting that the computation of $P[\tilde{u}_x]$ was not explicitly required, as we instead used the corresponding values of the auxiliary variable $\tilde{w}$ to evaluate the errors $\|u_x - P[\tilde{u}_x]\|$ due to Remark \ref{rmk:doubleu}.

\subsection{Energy-conservative fully discrete scheme}

The semi-discrete system (\ref{eq:mixfem1})--(\ref{eq:mixfem3}) is a system of ordinary differential equations of the form $\bM\by'=\bF$, where $$\bM = \begin{pmatrix}
    \bA & \bB\\
    \bB & \bA
\end{pmatrix}\ ,$$
$\bA$ and $\bB$ matrices with $\bA_{ij}=(\varphi_i,\varphi_j)$ and $\bB_{ij}=(\varphi_i,\varphi_j')$, and $\bF=\bF(\by)$ (see also the proof of Theorem \ref{thrm:Theorem1}). To integrate this particular system numerically, we first need to invert the matrix $\bM$. 

In case $\bA$ and $\bB$ are circulant, a linear system of the form
$\bM\bX=\bF$ with unknown vector $\bX=(\bx,\by)^T$ can be solved using the Fast Fourier Transform as a system of two equations with two unknowns $\bx$ and $\by$. For example, in the case of cubic splines or piecewise linear elements, denote the Discrete Fourier Transform matrix and its inverse (implemented via the Fast Fourier Transform) by $\mathcal{F}$ and $\mathcal{F}^{-1}$. Also, consider $\ba, \bb$ be the first rows of the matrix $\bA$ and $\bB$, respectively. Then, the algorithm for the solution of our system is presented in Algorithm \ref{alg:blockcirc}.

\begin{algorithm}
\begin{algorithmic}
\State Set $\mathcal{D}_\bA=\diag(\mathcal{F}\ba)$ and $\mathcal{D}_\bB=\diag(\mathcal{F}\bb)$
\State Set $\hat{\bz}=\mathcal{F}^{-1}\bz$
\State Compute $\hat{\by}=(\mathcal{D}_\bB \mathcal{D}_\bA^{-1}\mathcal{D}_\bB-\mathcal{D}_\bA)^{-1} \mathcal{D}_\bB\mathcal{D}_\bA^{-1}\hat{\bz}$ 
\State Compute $\hat{\bx}=\mathcal{D}_\bA^{-1}(\hat{\bz}-\mathcal{D}_\bB \hat{\by})$
\State The solution is $\bx=\mathcal{F}\hat{\bx}$ and $\by=\mathcal{F}\hat{\by}$. 
\end{algorithmic}\caption{Solution of block circulant system using FFT}\label{alg:blockcirc}
\end{algorithm}

In the case of Lagrange elements with $k>1$, the matrices $\bA$ and $\bB$ are banded but not circulant, and thus the inversion of the matrix $\bM$ can be performed efficiently using the block Sherman-Morrison-Woodbury (SMW) iterative method for nearly circulant matrices of \cite{Mitsotakis2024}. 

Assume that we have inverted the matrix $\bM$ and denote the resulting initial-value problem as
\begin{equation}\label{eq:ode}
\begin{aligned}
\frac{d}{dt}\by(t)&=\btf (t,\by(t)),\quad t\in (0,T],\\
\by(0)&=\by_0,
\end{aligned}
\end{equation}
where $\by$ is the unknown vector function. In order to preserve the energy functional $\mathcal{E}$ we employ the so-called Relaxation Runge--Kutta methods \cite{KE2019,RK2020}. For a uniform timestep $\Delta t\in (0,1)$ we consider a uniform grid $0=t_0<t_1<\cdots<t_K=T$ with $t_{i+1}=t_i+(i+1)\Delta t$ for all $i=0,1,\dots,K-1$. We also consider an explicit Runge--Kutta method with $s$ stages described by the Butcher tableau
\begin{align}\label{eq:butcher}
\begin{array}{c|c}
\bc & \bA\\
\hline
& \bb^\mathrm{T}
\end{array},
\end{align}
where $\bA=[a_{ij}]_{i,j=1}^s$ is an $s\times s$ lower-triangular matrix with zeros
in the principal diagonal, and $\bb=[b_j]_{j=1}^s$ and $\bc=[c_j]_{j=1}^s$ are
$s$-dimensional vectors. Note that here the matrix $\bA$ and the vector $\bb$ are different from those in Algorithm \ref{alg:blockcirc}. Let $\by^n$ be an approximation of $\by(t^n)$,
then the explicit Runge--Kutta method that corresponds to the Butcher tableau (\ref{eq:butcher})
can be expressed as
\begin{align}
\tilde{\by}^i&=\by^n+\Delta t\sum_{j=1}^{i-1} a_{ij}\,\btf(t_n+c_j\Delta t, \tilde{\by}^j), \quad i=1,2,\ldots,s, \label{eq:stages}\\
\by(t_n+\Delta t)\approx \by^{n+1}&=\by^n+\Delta t\sum_{i=1}^sb_i \,\btf(t_n+c_i\Delta t,\tilde{\by}^i). \label{eq:updates}
\end{align}

\begin{figure}[t]
\centering
\includegraphics[width=1\columnwidth]{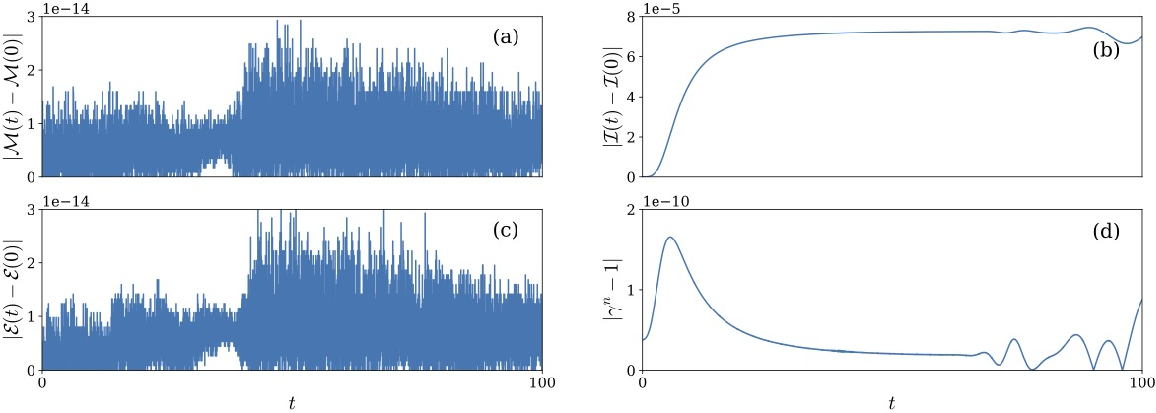}
\caption{Variation of (a) mass $\mathcal{M}$, (b) impulse $\mathcal{I}$, (c) energy $\mathcal{E}$, and (d) relaxation parameter $\gamma^n$ for $k=1$}
\label{fig:conserve}
\end{figure}

Let
$$\bd^n=\sum_{i=1}^s b_i \,\btf_i,$$
with $\btf_i=\btf(t_n+c_i\Delta t,\tilde{\by}^i)$, then the corresponding
relaxation Runge--Kutta method is formulated by the replacement of the update
formula (\ref{eq:updates}) with an update in the same direction as the previous
formula but of a different length:
\begin{align}\label{eq:rrkupdate}
\by(t_n+\gamma^n\Delta t)\approx \by_\gamma^{n+1}=\by^n+\gamma^n\Delta t \bd^n,
\end{align}
where $\by^n=\by^{n}_\gamma$, for $n=1,2,\ldots$.
The parameter $\gamma^n$ is called the relaxation parameter such that
$$\mathcal{E}(t^{n}+\gamma^n\Delta t; \by^{n+1}_\gamma)=\mathcal{E}(t^n;\by^n),$$
where $\mathcal{E}$ is the energy functional (\ref{eq:energyc}). The resulting equation is a quadratic equation in terms of $\gamma^n \Delta t$, which can be solved efficiently using Newton's method to avoid possible catastrophic cancellation errors \cite{Mitsotakis2023, Mitsotakis2025}.
The relaxation Runge--Kutta method is practically adapting the time step $\Delta t$ at every timestep to $\Delta t^n_\gamma=\gamma^n\Delta t$.

To demonstrate the approximation properties of the energy-conservative method we considered the initial-periodic boundary value problem (\ref{eq:ipbvp}) in $[-100,100]$, with a generic initial condition $U_0(x)=e^{-x^2/10}$ for times up to $T=100$. Such an initial condition is not a traveling wave, and thus the method will be unbiased during the computation of the conserved quantities. For this experiment we used $h=0.1$ and $\Delta t=0.01$.

Figures \ref{fig:conserve} (a), (b), and (c) illustrate the absolute errors in mass, impulse, and energy, respectively. As anticipated, the mass and energy are conserved up to machine precision, while the error in impulse is $O(10^{-5})$. In Figure \ref{fig:conserve}(d) we observe that the relaxation parameter $\gamma_n$ is $1+O(10^{-10})$, verifying again the theoretical estimates of \cite{KE2019}. Similar results were observed for $k>1$, but we omit them for brevity. 

In conclusion, we validate the convergence rate of the Impulse functional $\mathcal{I}$. We use the same initial condition $U_0$ as before in the interval $[-50, 50]$, and compute the errors $\max_{t_n} |\mathcal{I}(t_n; \tilde{u}) - \mathcal{I}(0; \tilde{u})|$ for $h = 1/N$ with $N = 100, 200, 500, 800, 1000$ for $k = 1, 2, 3$, and $N = 100, 200, 400, 500, 800$ for $k = 4$ up to $T=10$. For $k = 1, 2, 3$, we use $\Delta t = 0.01$, while for $k = 3$ and $k = 4$, we use $\Delta t = 0.001$ and $\Delta t = 0.0005$, respectively. The results shown in Figure~\ref{fig:ratesI} present logarithmic plots of the errors versus $h$ and clearly confirm the predicted convergence-rate pattern of Proposition~\ref{prop:impulsec} for $k=1,2,3,4$. For $k>4$, the errors are already at the level of machine precision, making it difficult to experimentally verify convergence rates of order $8$.

\begin{figure}[t]
\centering
\includegraphics[width=0.9\columnwidth]{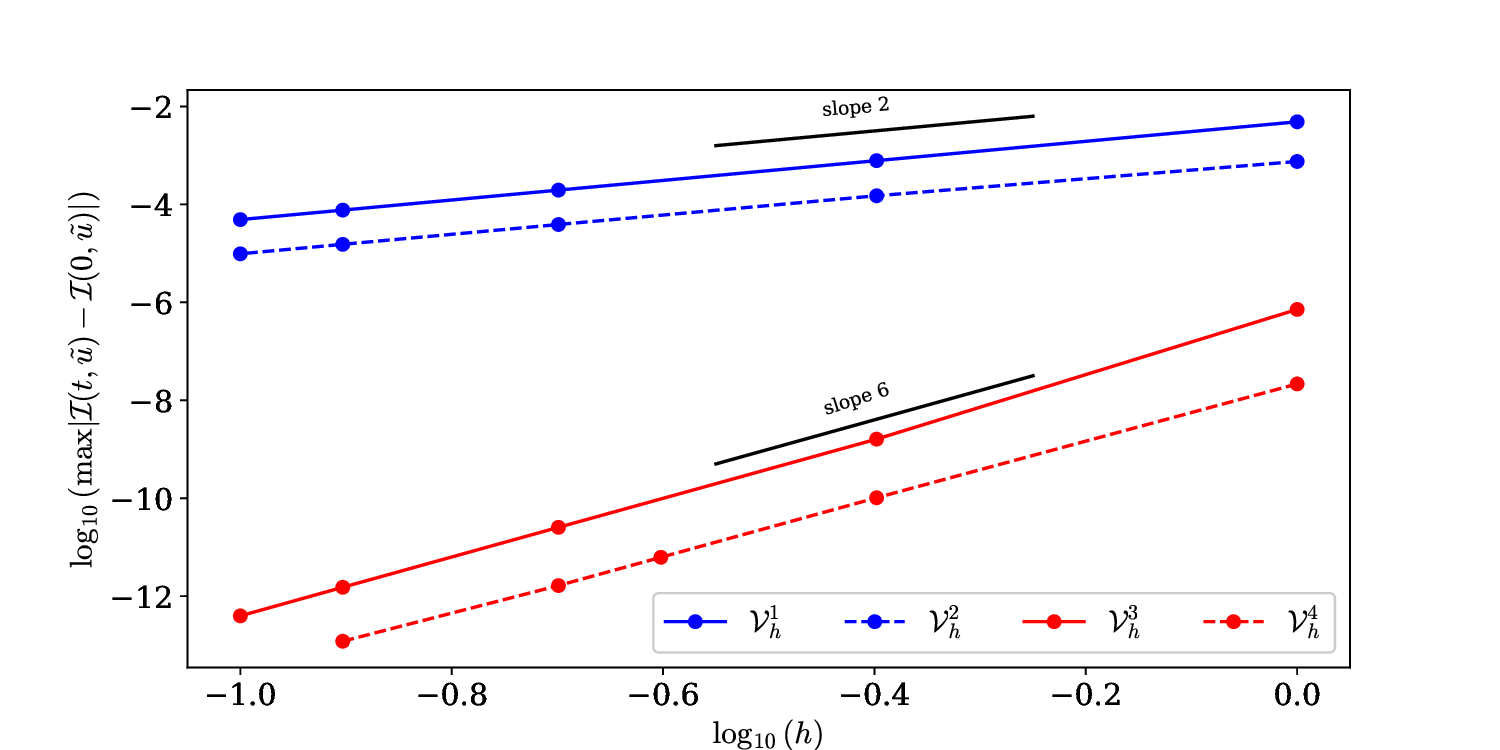}
\caption{Experimental convergence rates of impulse $\mathcal{I}$ for $k=1,2,3,4$}
\label{fig:ratesI}
\end{figure}

\section{Conclusions}\label{sec:conclusions}

In this work, we proved that the $L^2$ projection of sufficiently smooth and periodic function $u$ onto Lagrange finite element spaces of periodic piecewise polynomial functions satisfies
$\|P[(Pu-u)_x]\| = O(h^4)$ if $k=1$, $\|P[(Pu-u)_x]\| =O(h^{k+1})$ if $k=2\kappa+1$, $\kappa\ge 1$, and $\|P[(Pu-u)_x]\| =O(h^{k})$ if $k=2\kappa$, $\kappa\ge 1$.
This approximation property of finite element spaces directly influences the convergence of energy-conservative finite element methods applied to certain nonlinear and dispersive wave equations. We illustrate this effect by deriving {\em a priori} error estimates of a finite element method for the regularized long-wave equation that conserves a cubic energy functional and the linear mass functional. The theoretical findings are further confirmed by numerical experiments.

\bibliographystyle{plain}
\bibliography{biblio}
\bigskip

\end{document}